\documentclass[11pt]{article}

\usepackage{a4wide,amsfonts,amsmath,latexsym,amsthm,amssymb,euscript,graphicx,units,mathrsfs,color}
\usepackage{amsmath}
\usepackage{amssymb}
\usepackage{amsthm}
\usepackage{latexsym}
\usepackage{color}
\usepackage{enumerate}
\usepackage{graphicx}
\usepackage{color} 

\usepackage[T1]{fontenc}
\usepackage[english]{babel}
\usepackage{geometry}
\DeclareSymbolFont{calletters}{OMS}{cmsy}{m}{n}
\DeclareSymbolFontAlphabet{\mathcal}{calletters}
\usepackage{float}
\newfloat{figure}{h}{lof}
\floatname{figure}{\figurename}
\def\be{\begin{eqnarray}}
\def\Ee{\end{eqnarray}}

\def\b*{\begin{eqnarray*}}
\def\E*{\end{eqnarray*}}

\newcommand{\real}{\mathbb{R}}

\newtheorem{Theorem}{Theorem}[part]

\newtheorem{Proposition}{Proposition}[part]

\newtheorem{Lemma}{Lemma}[part]
\newtheorem{Corollary}{Corollary}[part]
\newtheorem{Remark}{Remark}[part]

\makeatletter \@addtoreset{equation}{section}

\@addtoreset{Definition}{section}

\@addtoreset{Theorem}{section}

\@addtoreset{Proposition}{section}

\@addtoreset{Property}{section}

\@addtoreset{Assumption}{section}

\@addtoreset{Corollary}{section}

\@addtoreset{Lemma}{section}

\@addtoreset{Remark}{section}

\@addtoreset{Example}{section}

\@addtoreset{Properties}{section}

\@addtoreset{footnote}{page}

\newcommand{\No}[1]{\left\|#1\right\|}     
\newcommand{\abs}[1]{\left|#1\right|}     

\def \D{\mathbb{D}}
\def \E{\mathbb{E}}
\def \F{\mathbb{F}}
\def \H{\mathbb{H}}

\def \N{\mathbb{N}}
\def \P{\mathbb{P}}

\def \R{\mathbb{R}}

\def\Fc{{\cal F}}

\def\Kc{{\cal K}}

\def \Sum{\displaystyle\sum}

\def\={\;=\;}
\def\.{\;.}

\def\reff#1{{\rm(\ref{#1})}}

\def \i{1\!\mbox{\rm I}}
\def\1{{\bf 1}}

 \def\normeL2#1{\left\|{#1}\right\|_{L^2}}

\def\E{\mathbb{E}}
\def\P{\mathbb{P}}
\def\h{\mathfrak{H}}
\def\S{\mathbb{S}}
\def\H{\mathbb{H}}

\newcommand\tr[1]{\circ \tau_{#1}}
\def\dh{\dot{h}}

\allowdisplaybreaks

\def\b{\textcolor{blue}}

\author{Thibaut Mastrolia \footnote{Universit\'e Paris-Dauphine, CEREMADE UMR CNRS 7534, Place du Mar\'echal De Lattre De Tassigny, 75775 Paris Cedex 16, FRANCE, \texttt{mastrolia@ceremade.dauphine.fr}} \and Dylan Possama\"i\footnote{Universit\'e Paris-Dauphine, CEREMADE UMR CNRS 7534, Place du Mar\'echal De Lattre De Tassigny, 75775 Paris Cedex 16, FRANCE, \texttt{possamai@ceremade.dauphine.fr}} \and Anthony R\'eveillac\footnote{IMT UMR CNRS 5219, Universit\'e de Toulouse, INSA de Toulouse, 135 avenue de Rangueil, 31077 Toulouse Cedex 4, FRANCE,  \texttt{anthony.reveillac@insa-toulouse.fr} }}

\title{On the Malliavin differentiability of BSDEs}

\begin{document}

\maketitle

\begin{abstract}
In this paper we provide new conditions for the Malliavin differentiability of solutions of Lipschitz or quadratic BSDEs. Our results rely on the interpretation of the Malliavin derivative as a G\^ateaux derivative in the directions of the Cameron-Martin space. Incidentally, we provide a new formulation for the characterization of the Malliavin-Sobolev type spaces $\D^{1,p}$. 
\end{abstract}

\vspace{1em} 
{\noindent \textit{Key words:} Malliavin's calculus; abstract Wiener space; BSDEs.
}

\vspace{1em}
\noindent
{\noindent \textit{AMS 2010 subject classification:} Primary: 60H10; Secondary: 60H07.
\normalsize
}

\section{Introduction}

Backward Stochastic Differential Equations (BSDEs) have been studied extensively in the last two decades as they naturally arise in the context of stochastic control problems (for instance in Finance see \cite{EPQ}), and as they provide a probabilistic representation for solution to semi-linear parabolic PDEs, via a non-linear Feynman-Kac formula (see \cite{pardouxpeng}). Before going further let us recall that this class of equations has been introduced in \cite{Bismut_78,Pardoux_Peng90,pardouxpeng} and that a BSDE can be formulated as:
\begin{equation}
\label{eq:BSDE}
Y_t = \xi +\int_t^T f(s,Y_s,Z_s)ds -\int_t^T Z_s dW_s, \quad t\in [0,T],
\end{equation} 
where $T$ is a fixed positive number, $W:=(W_t)_{t\in [0,T]}$ is a one-dimensional Brownian motion defined on a probability space $(\Omega,\mathcal{F}_T,\P)$ with natural filtration $(\mathcal{F}_t)_{t\in [0,T]}$. The data of the equation are the $\mathcal{F}_T$-measurable r.v. $\xi$, called the terminal condition, and the mapping $f:[0,T] \times \Omega \times \real^2 \longrightarrow \real$ which is a progressively measurable process and where according to the notations used in the literature we write $f(t,y,z)$ for $f(t,\omega,y,z)$. A solution to the BSDE \eqref{eq:BSDE} is then a pair of predictable processes $(Y,Z)$, with appropriate integrability properties, such that Relation \eqref{eq:BSDE} holds $\P-$a.s.

\vspace{0.5em}
When dealing with applications, one needs to obtain regularity properties on the solution $(Y,Z)$, such as the Malliavin differentiability of the random variables $Y_t$, $Z_t$ at a given time $t$ in $[0,T]$. Note that for the $Z$ component this question needs to be clarified a little bit because of the definition of $Z$, cf. Theorem \ref{thm:yzd12} for a precise statement. More precisely, one needs to answer the following question: 
$$\textrm{Which conditions on the data $\xi$ and $f$ in \eqref{eq:BSDE} ensure that $Y_t$, $Z_t$ are Malliavin differentiable?}$$ 
This question was first addressed in the paper \cite{pardouxpeng} in a Markovian setting, that is when $\xi:=g(X_T)$ and $f(t,\omega,y,z):=h(t,X_t(\omega),y,z)$ where $g:\real \to \real$ and $h:[0,T]\times \real^3 \to \real$ are regular enough deterministic functions and $X:=(X_t)_{t\in [0,T]}$ is the unique solution to a SDE of the form:
$$ X_t=X_0+\int_0^t \sigma(s,X_s) dW_s +\int_0^t b(s,X_s) ds, \quad t\in [0,T], $$
with regular enough coefficients $\sigma, b:[0,T]\times \real \longrightarrow \real$. In that framework, Pardoux and Peng proved in \cite[Proposition 2.2]{pardouxpeng} that, under (essentially) the following conditions:
\begin{itemize}
\item[(PP1)] $g$ is continuously differentiable with bounded derivative.
\item[(PP2)] $h$ is continuously differentiable in $(x,y,z)$ with bounded derivatives uniformly in time, 
\end{itemize}
$Y_t$ is Malliavin differentiable at any time $t$, with a similar statement for $Z$, and the Malliavin derivatives of $Y$ and $Z$ provide a solution to an explicit linear BSDE. To be more precise, in \cite{pardouxpeng} the authors make one assumption for the whole paper which is stronger than (PP1)-(PP2) above. However a careful reading of the proof of \cite[Proposition 2.2]{pardouxpeng} enables one to conclude that Conditions (PP1)-(PP2) are sufficient to obtain the Malliavin differentiability of the solution. Assumptions (PP1)-(PP2) look pretty intuitive since they basically require the Malliavin differentiability of the terminal condition $\xi$ and of the generator $f$ once the component $(y,z)$ are frozen, i.e., of the process $(t,\omega) \longmapsto f(t,\omega,y,z)$ for given $(y,z)$. Hence, it is natural to expect that the latter conditions can be easily generalized to the non-Markovian framework. Unfortunately, the first result in that direction, given by El Karoui, Peng and Quenez in \cite{EPQ}, requires more stringent conditions than the aforementioned intuitive ones. More explicitly, the main result in \cite{EPQ} concerning the Malliavin differentiability of the solution to the BSDE \eqref{eq:BSDE} (essentially) involves the following conditions (see \cite[Proposition 5.3]{EPQ} for a precise statement):
\begin{itemize}
\item[(EPQ1)] $\xi$ is Malliavin differentiable\footnote{\textit{i.e.} $\xi$ is in $\D^{1,2}$} and $\E[|\xi|^4]<+\infty.$
\item[(EPQ2)] At any time $t\in[0,T]$, the r.v. $\omega \longmapsto f(t,\omega,Y_t,Z_t)$ is Malliavin differentiable\footnote{In fact as an adapted process it belongs to $\D^{1,2}$, we refer to the space $\mathbb{L}^a_{1,2}$ whose precise definition is recalled in \cite[p.~58]{EPQ}} with Malliavin derivative denoted by $D_\cdot f(t,Y_t,Z_t)$ such that there exists a predictable process $K^{\theta}:=(K_t^\theta)_{t\in [0,T]}$ with $\int_0^T \E[(\int_0^T|K_s^\theta|^2ds)^2] d\theta<+\infty$, and such that for any $(y_1,y_2,z_1,z_2) \in \real^4$ it holds for a.e. $\theta \in [0,T]$ that:
$$ |D_\theta f(t,\omega,y_1,z_1) - D_\theta f(t,\omega,y_2,z_2)| \leq K^\theta_t(\omega) (|y_1-y_2|+|z_1-z_2|).$$
\end{itemize}      
Roughly speaking, this means that $\xi$ and $\omega \longmapsto f(t,\omega,y,z)$ have to be Malliavin differentiable, but in order to prove that $Y$ and $Z$ are Malliavin differentiable, one needs to enforce an extra regularity conditions on each of the data: that is $\xi$ has a finite moment of order $4$, and the Malliavin derivative of the driver $f$ is Lipschitz continuous in $(y,z)$ with a sufficiently integrable stochastic Lipschitz constant $K$. Note that a careful reading of the proof allows one to conclude that the moment conditions on $\xi$ and $Df$ can actually be relaxed to hold only in $L^{2+\varepsilon}$ for some $\varepsilon>0$. Besides, as noted in \cite[Remark at the bottom of p.~59]{EPQ}, if $K$ is bounded then the proof can be modified so that the extra integrability condition on $\xi$ (i.e. $\E[|\xi|^4]<+\infty$) can be dropped. However, even in this case, one can check that in the Markovian framework, Conditions (EPQ1)-(EPQ2) are strictly stronger than Conditions (PP1)-(PP2).

\vspace{0.5em}
Since these two seminal papers, the most notable extension was concerned with the study of the Malliavin differentiability of $(Y,Z)$ in a quadratic setting, that is to say when the generator $f$ has quadratic growth in the $z$ variable, a problem addressed in \cite{ank,DosReis_Imkeller, dosreis,irr}. Notice nonetheless that the proofs in these references are strongly influenced by the ones in the Lipschitz setting of \cite{pardouxpeng,EPQ}, as they all start by approximating the quadratic generators by Lipschitz ones, to which they apply the results of \cite{pardouxpeng,EPQ}. The applications of the Malliavin differentiability of BSDEs also received a lot of attention in the literature. Hence, it was used in the context of numerical schemes for BSDEs in, among others, \cite{briand_labart,hns}, or to study the existence and regularity of densities for the marginal laws of $(Y,Z)$ in \cite{ab,ak,mpr}. However, in all the above references, the authors always refer to either \cite{pardouxpeng,EPQ} in a Lipschitz context or \cite{ank} in a quadratic context, when stating differentiability results in the Malliavin sense (see for instance the sentence before Theorem 2.2 in \cite{ak}, or Step 2 in the proof of Theorem 3.3 in \cite{ab}, which refers to \cite{ak}, or the proof of Part a) of Theorem 2.6 in \cite{hns}, or Proposition 3.2 in \cite{mpr}). \vspace{0.5em}

The aim of this paper is to provide an alternative sufficient condition to (EPQ1)-(EPQ2) for the Malliavin differentiability of the solution to a BSDE of the form \eqref{eq:BSDE} in the general non-Markovian setting. Our main result in that direction is Theorem \ref{thm:yzd12} below, using a fundamentally different approach from \cite{EPQ,pardouxpeng}, as well as different type of assumptions. Since they involve some notations concerning the analysis on the Wiener space, we refrain from detailing them immediately, and rather explain informally what are the main differences between our approach and the one of \cite{EPQ}. A natural way to solve a BSDE of the form \eqref{eq:BSDE} when the driver $f$ is Lipschitz in $(y,z)$ is to make use of a Picard iteration, that is to say a family $(Y^n,Z^n)$ of solutions to BSDEs satisfying\begin{equation}
\label{eq:BSDEn}
Y^n_t =\xi +\int_t^T f(s,Y_s^{n-1},Z_s^{n-1}) ds -\int_t^T Z_s^n dW_s, \quad t\in [0,T],
\end{equation}
where $Y^0\equiv Z^0\equiv 0$. Then, a fixed point argument allows one to construct, in appropriate spaces, a solution $(Y,Z)$ to Equation \eqref{eq:BSDE}. If $\xi$ and $f(t,y,z)$ are Malliavin differentiable, then so is $(Y^n,Z^n)$. Then, it just remains to prove that this property extends to the limits $Y$ and $Z$ of respectively $Y^n$ and $Z^n$, in appropriate spaces. More precisely this is done by a uniform (in $n$) control of the Sobolev norms of $Y^n,Z^n$ or equivalently by proving that the Malliavin derivatives $(DY^n,DZ^n)$ of $(Y^n,Z^n$) converge to the solution of a linear BSDE whose solution will be the Malliavin derivatives $(DY, DZ)$ of $Y$ and $Z$. This last step is exactly where the extra regularity (EPQ1)-(EPQ2) is needed. It appears quite clearly that for this approach, the conditions of \cite{EPQ} cannot be optimized in the general case. Even though this idea seems pretty natural, it is based on a choice somehow arbitrary. Indeed, a necessary condition for $DY_t$ to be well-defined at a given time $t$, is that there exists a sequence of random variables $(F^n)_n$ converging to $Y_t$ in $L^2$ such that each variable $F^n$ is Malliavin differentiable with derivative $DF^n$ and such that $DF^n$ converges, with respect to a suitable norm, to $DY_t$. As a consequence, in the approach described above, one believes that this sequence $(F^n)_n$ can be chosen to be the Picard iteration $(Y^n)_n$. Once again, this idea looks very natural, according to the same type of proofs for SDEs, but then one sees that in the BSDE framework this intuitive idea leads to pretty heavy assumptions. We elaborate a little bit more on this point in Section \ref{section:discu}.

\vspace{0.5em}
Regarding the discussion above, one could think of trying to find a sequence of processes known to approximate the Malliavin derivative of $Y$ (and $Z$) when $Y$ is Malliavin differentiable. This approximation is provided by the well-known interpretation of the Malliavin derivative as a G\^ateaux derivative in the directions of the Cameron-Martin space. More precisely, a necessary condition for $Y_t$ to belong to $\D^{1,2}$, is that for any absolutely continuous function $h$ starting from $0$ at $0$ with derivative denoted $\dot{h}$, the difference quotient $\varepsilon^{-1}(Y_t(\omega+\varepsilon h)-Y_t(\omega))$ converges, in a sense to be made precise, as $\varepsilon$ goes to $0$ to $\langle D Y_t,\dot{h}\rangle_{L^2([0,T])}$. This fact was initially given by Malliavin and then extended by Stroock, Shigekawa, Kusuoka and Sugita in a series of papers \cite{Malliavin,stroock,Shige,Kusuoka,Sugita}. In addition, Sugita proved in \cite{Sugita} that a r.v. $F$ is Malliavin differentiable if it is \textit{ray absolutely continuous}\footnote{we refer to Section \ref{section:Malliavin} where this notion is recalled} and if it is stochastically G\^ateaux differentiable. Using the main ideas of \cite{Sugita} we provide incidentally a new formulation of the characterization of the Malliavin-Sobolev type spaces $\D^{1,p}$ in Theorem \ref{th:newcarD}. Since we did not find explicitly this characterization in the literature, we believe that this result is new and maybe interesting by itself. The main point is that this formulation is especially handy when dealing with stochastic equations like BSDEs. With this result at hand, we obtain new conditions (see Assumptions (D), ($H_1$) and ($H_2$) at the beginning of Section \ref{section:BSDEs}) for $Y,Z$ to be Malliavin differentiable, see Theorem \ref{thm:yzd12}. Our assumptions refine those of \cite{pardouxpeng,EPQ} in the Markovian case, and our approach is directly applicable to quadratic growth BSDEs since we do not rely on any approximation procedure. We refer the reader to Section \ref{section:applidiscus} for some examples and a discussion on the differences between our approach and the one of \cite{pardouxpeng,EPQ}.

\vspace{0.5em}
The rest of the paper is organised as follows. We start below with some preliminaries. Then we turn in Section \ref{section:Wiener} to some elements of analysis on the Wiener space. Our characterization of the sets $\D^{1,p}$ is given in Section \ref{section:Malliavin}, and the material on the Malliavin differentiability of BSDEs itself is contained in Section \ref{section:BSDEs}. We provide applications and a comparison of the results in Section \ref{section:applidiscus}. Finally, we extend our approach to quadratic growth BSDEs in Section \ref{section:quadratic}.

\section{Preliminaries}
\label{sec:preli}

\subsection{Notations}

We fix throughout the paper a time horizon $T>0$ and $d$ a positive integer. For any positive integer $k$, we denote by $\|\cdot\|$ the Euclidian norm in $\R^{k}$ and by $\cdot$ the inner product, without mention of $k$ which will be clear in the context. For any positive integers $n$ and $m$, we identify $\mathbb R^{n\times m}$ with the space of real matrices with $n$ rows and $m$ columns, endowed with the Euclidean norm on $\R^{n\times m}$. Let $M$ be in $\R^{n\times m}$, $1\leq j\leq n$ and $1\leq \ell\leq m$, we denote by $M^{j,:}\in \R^{1\times m}$ (resp. $M^{:,\ell}\in \R^{n,1}$) its $j$-th row (resp. its $\ell$-th column). We set $M^\top\in \R^{m\times n}$ to be the transpose of $M$. We also identify $\R^k$ with $\R^{1,k}$. Let now $\Omega:=C_0([0,T],\R^d)$ be the canonical Wiener space of continuous function $\omega:=(\omega^1,\ldots,\omega^d)^\top$ from $[0,T]$ to $\R^d$ such that $\omega(0)=(0,\ldots,0)^\top$. Let $W:=(W_t^1,\ldots,W_t^d)^\top_{t\in [0,T]}$ be the canonical Wiener process, that is, for any time $t$ in $[0,T]$, $W_t$ denotes the evaluation mapping: $W_t^i(\omega):=\omega_t^{i}$ for any element $\omega$ in $\Omega$ and $i$ in $\{1,\ldots,d\}$. We set $\F^o$ the natural filtration of $W$. Under the Wiener measure $\P_0$, the process $W$ is a standard Brownian motion and we denote by $\F:=(\mathcal F_t)_{t\in[0,T]}$ the usual augmentation (which is right-continuous and complete) of $\F^o$ under $\P_0$. Unless otherwise stated, all the expectations considered in this paper will have to be understood as expectations under $\P_0$, and all notions of measurability for elements of $\Omega$ will be with respect to the filtration $\F$ or the $\sigma$-field $\mathcal F_T$. 

\vspace{0.5em}
For any Hilbert space $\mathcal K$, for any $p\geq 1$ and for any $t\in [0,T]$, we set $L^p([t,T];\mathcal K)$ to be following space
$$L^p([t,T];\mathcal K):=\left\{f:[t,T]\longrightarrow\mathcal K, \text{ Borel-measurable, s.t. }\int_t^T \| f(s)\|^p_{\mathcal{K}}ds <+\infty \right\},$$
where the norm $\No{\cdot}_{\Kc}$ is the one canonically induced by the inner product on $\Kc$. We denote, for simplicity, by $\h:=L^2([0,T];\R^{ d})$ and by $\langle \cdot, \cdot \rangle_\h$ its canonical inner product, that is to say
$$\langle f, g\rangle_\h:=\int_0^Tf(s) \cdot g(s)ds = \sum_{i=1}^d \int_0^Tf^i(s) g^i(s) ds,\ (f,g)\in \h^2.$$
Let now $H$ be the Cameron-Martin space that is the space of functions in $\Omega$ which are absolutely continuous with square-integrable derivative and which start from $0$ at $0$:
$$H:=\left\{ h:[0,T] \longrightarrow \real^d, \; \exists \dot{h}\in\h, \; h(t)=\int_0^t \dot{h}(x)dx, \; \forall t\in [0,T]\right\},$$
For any $h$ in $H$, we will always denote by $\dot{h}$ a version of its Radon-Nykodym density with respect to the Lebesgue measure. 
Then, $H$ is an Hilbert space equipped with the inner product $\langle h_1,h_2 \rangle_{H}:=\langle \dot{h_1},\dot{h_2} \rangle_{\h}$, for any $(h_1, h_2)\in H\times H$, and with associated norm $\|h\|_H^2:=\langle \dot h, \dot h \rangle_\h$.
Define next $L^p(\Kc)$ as the set of all $\mathcal F_T$-measurable random variables $F$ which are valued in an Hilbert space $\mathcal{K}$, and such that $\No{F}_{L^p(\Kc)}^p<+\infty$, where
$$\No{F}_{L^p(\Kc)}:=\left(\E\left[\No{F}_{\mathcal{K}}^p\right]\right)^{1/p}.$$ 

\vspace{0.5em}

Let now $\mathcal{S}$ be the set of cylindrical functionals, that is the set of $\mathbb R$-valued random variables $F$ of the form
\begin{equation}
\label{eq:cylindrical}
F=f(W(h_1),\ldots,W(h_n)), \quad (h_1,\ldots,h_n) \in H^n, \; f \in C^\infty_b(\real^n), \text{ for some }n\geq 1,
\end{equation}
where $W(h):=\int_0^T \dot{h}_s\cdot  dW_s:= \sum_{i=1}^d \int_0^T \dot{h}_s^{i} dW_s^{i}$ for any $h$ in $H$ and where $C^\infty_b(\real^n)$ denotes the space of bounded mapping which are infinitely continuously differentiable with bounded derivatives. For any $F$ in $\mathcal S$ of the form \eqref{eq:cylindrical}, the Malliavin derivative $\nabla F$ of $F$ is defined as the following $H$-valued random variable:
\begin{equation}
\label{eq:DF}
\nabla F:=\sum_{i=1}^n f_{x_i}(W(h_1),\ldots,W(h_n)) h_i,
\end{equation}
where $f_{x_i}:=\frac{df}{dx_i}$.
It is then customary to identify $\nabla F$ with the stochastic process $(\nabla_t F)_{t\in [0,T]}$. More precisely, we define for any $(t,\omega)\in [0,T]\times \Omega,$
$$\nabla_t F(\omega):=\sum_{i=1}^n f_{x_i}\left(W(h_1)(\omega),\ldots,W(h_n)(\omega)\right) h_i(t).$$  Denote then by $\mathbb{D}^{1,p}$ the closure of $\mathcal{S}$ with respect to the Malliavin-Sobolev semi-norm $\|\cdot\|_{1,p}$, defined as:
$$ \|F\|_{1,p}:=\left(\E\left[|F|^p\right] + \E\left[\|\nabla F\|_{H}^p\right]\right)^{1/p}. $$ We set $\D^{1,\infty}:= \bigcap_{p\geq 2} \D^{1,p}$.
In order to link our notations with the ones of the related papers \cite{pardouxpeng,EPQ} we make use of the notation $DF$ to represent the derivative of $\nabla F$ as: 
$$ \nabla_t F=\int_0^t D_s F ds, \quad t\in [0,T]. $$ 

We denote by $\delta: L^p(H)\longrightarrow L^p(\R)$ the adjoint operator of $\nabla$ by the following duality relationship:
$$ \E[F\delta(u)]=\E[\langle \nabla F, u\rangle_{H}],\quad \forall u\in \text{dom}(\delta), \quad \textrm{where}$$
$$ \text{dom}(\delta):=\left\{u\in L^p(H), \;\exists c_u>0, \; |\E[\langle \nabla F,u\rangle_{H}]| \leq c_u\| F\|_{L^p(\R)}, \; \forall F\in \D^{1,p} \right\}.$$
$\delta$ is also known under the name of Skorohod (or divergence) operator. Recall that any element $u$ of the form $u:=G h$ with $G$ in $\mathcal{S}$ and $h$ in $H$ belongs to ${\rm{dom}}(\delta)$ and that
\begin{equation}
\label{eq:deltaprod}
\delta(G h)= G W(h) -\langle \nabla G, h\rangle_H,
\end{equation}
see for example \cite[Relation (1.46)]{Nualartbook}. Note that for any $h$ in $H$, $\delta(h)=W(h)$.

\vspace{0.5em}
Notice that in \cite{Sugita} the cylindrical space, that we will denote by $\mathcal P$ in the following, is the space of functionals $F$ of the form \eqref{eq:cylindrical} with $f$ a polynomial. More precisely let $\mathcal{P}$ be the set of polynomial cylindrical functionals, that is the set of random variables $F$ of the form
\begin{equation}
\label{eq:cylindricalP}
F=f(W(h_1),\ldots,W(h_n)), \quad (h_1,\ldots,h_n) \in H^n, \; f \in \R^n[X], \text{ for some }n\geq 1,
\end{equation}
where $\R^n[X]$ denotes the set of polynomials of degree less or equal to $n$. However, the closures of both $\mathcal S$ and $\mathcal P$ with respect to any $\|\cdot\|_{1,p}$ coincide, as any polynomial together with its derivative can be approximated in $L^p(\real^n)$ (see Lemma \ref{lemma:tec1} below).

\begin{Lemma}
\label{lemma:tec1}
Let $G$ be in $\mathcal{P}$. There exists a sequence $(G^N)_{N\geq 1} \subset \mathcal{S}$ such that $\lim_{N\to+\infty} G^N = G$ in $\mathbb{D}^{1,r}$ for any $r \geq 1$.
\end{Lemma}
\begin{proof}
Let $G:=f(W(h_1),\cdots,W(h_n))$ with $n\geq 1$, $h_i$ in $H$ and $f$ in $\real^n[X]$. Without loss of generality, we assume that the family $(h_1,\ldots,h_n)$ is orthonormal in $H$. Let $\theta$ be a cutoff function, that is a mapping $\theta:\real^n \longrightarrow \real^+$ such that $\theta(x)=1$ if $\|x\|<1$, $\theta(x)=0$ for $\|x\|\geq 2$, and such that $\theta\in C_b^\infty(\real^n)$. For $N\geq 1$, we set:
$$ G^N:=f^N(W(h_1),\cdots,W(h_n)), \quad f^N(x):=f(x) \times \theta(x/N), \; x \in \real^n.$$
Note that each random variable $G^N$ belongs to $\mathcal{S}$. Fix $r\geq 1$. We aim in proving that $\lim_{N\to +\infty} \|G^N-G\|_{1,r}=0$. On the one hand, 
\begin{align*}
\E[|G^N-G|^r] &= \E\left[|G|^r \abs{\theta\left(\frac{W(h_1)}{N},\cdots,\frac{W(h_n)}{N}\right)-1}^r\right]\\
&\leq \E[|G|^{2r}]^{1/2} \E\left[\abs{\theta\left(\frac{W(h_1)}{N},\cdots,\frac{W(h_n)}{N}\right)-1}^{2r}\right]^{1/2}\\
&\leq C \int_{\real^n\setminus B^n(0,N)} e^{-\|x\|^2/2} dx,
\end{align*} where $C$ is a positive constant. Hence, $\lim\limits_{N\to +\infty}\E[|G^N-G|^r] =0$.
We now turn to the proof of the convergence of the derivatives. We have
$$ \nabla G^N = \sum_{i=1}^n \frac{\partial f^N}{\partial x_i}(W(h_1),\cdots,W(h_n)) h_i,$$
with $\frac{\partial f^N}{\partial x_i}(x) = \frac{\partial f}{\partial x_i} \theta(x/N) + N^{-1} f(x) \frac{\partial \theta}{\partial x_i}(x/N)$. Hence:
\begin{align*}
&\E\left[\|\nabla (G^N-G)\|_H^{2r}\right] = \sum_{i=1}^n \E\left[\abs{\frac{\partial f^N}{\partial x_i}-\frac{\partial f}{\partial x_i}}^{2r}(W(h_1),\cdots,W(h_n))\right]\\
&\leq C \left(\sum_{i=1}^n \E\left[\abs{\frac{\partial f}{\partial x_i}(W(h_1),\cdots,W(h_n))}^{2r} \abs{\theta\left(\frac{W(h_1)}{N},\cdots,\frac{W(h_n)}{N}\right)-1}^{2r}\right]\right.\\
&\hspace{0.9em} \left. + N^{-2r}\E\left[\abs{(\frac{\partial \theta}{\partial x_i} \times f)(W(h_1),\cdots,W(h_n))}^{2r} \right]\right)\\
&\underset{N\to +\infty}{\longrightarrow} 0.
\end{align*}
\qed
\end{proof}

\vspace{0.5em}
We conclude this section by introducing the following norms and spaces which are of interest when studying BSDEs. For any positive integers $p,n$, we set $\mathbb S_{n}^p$ the space of $\mathbb R^{ n}$-valued, continuous and $\mathbb F$-progressively measurable processes $Y$ s.t.
$$\No{Y}^p_{\mathbb S_{ n}^p}:=\mathbb E\left[\underset{0\leq t\leq T}{\sup} \|Y_t\|^p\right]<+\infty.$$
We denote by $\mathbb H_{n,d}^p$ the space of $\mathbb R^{n \times d}$-valued and $\F$-predictable processes $Z$ such that
$$\No{Z}^p_{\mathbb H_{n,d}^p}:=\mathbb E\left[\left(\int_0^T \sum_{j=1}^n \|Z_t^j\|^2 dt\right)^{\frac p2}\right]<+\infty.$$
We set $\mathbb S^p:=\mathbb S_{1}^p$ and $\H_{d}^p:=\H_{1,d}^p$.
\section{Some elements of analysis on the Wiener space}
\label{section:Wiener}

One of the main tool that we will use throughout this paper is the shift operator along directions in the Cameron-Martin space. More precisely, for any $h\in H$, we define the following shift operator $\tau_{h}:\Omega\longrightarrow\Omega$ by
$$\tau_{h}(\omega):=\omega + h:=(\omega^1+h^1,\ldots,\omega^d+h^d)^\top.$$
Note that the fact that $h$ belongs to $H$ ensures that $\tau_h$ is a measurable shift on the Wiener space. In fact, one can be a bit more precise, since according to \cite[Lemma B.2.1]{ustunelzakai} for any $\mathcal{F}_T$-measurable r.v. $F$ the mapping $h \longmapsto F \tr{h}$ is continuous in probability from $H$ to $L^0(\R^{ d})$, the space of real-valued and $\Fc_T$-measurable random variables, see Lemma \ref{lemma:ctsH} below. Taking $F={\rm Id}$, one gets that $\tau_h$ is a continuous mapping on $\Omega$ for any $h$ in $H$. We list below some other properties of such shifts.

\begin{Lemma}[Appendix B.2, \cite{ustunelzakai}]
\label{lemma:UstunelZakai}
Let $X$ and $Y$ be two $\mathcal F_T$-measurable random variables. If $X=Y$, $\P_0-$a.s., then for any $h$ in $H$, 
$$ X\tr{h}=Y\tr{h}, \ \P_0-a.s. $$
\end{Lemma}

We recall, the quite surprising result that any r.v. is continuous in probability in the directions of the Cameron-Martin space. More precisely: 

\begin{Lemma}[Lemma B.2.1, \cite{ustunelzakai}]
\label{lemma:ctsH}
Let $F$ be a $\mathcal{F}_T$-measurable random variable. The mapping $h\longmapsto F\tr{h}$ is continuous from $H$ to $L^0(\real^{d})$ where the convergence is in probability.  
\end{Lemma}

One of the main technique when working with shifts on the path space is the famous Cameron-Martin formula.
\begin{Proposition}\label{prop.cam}
$($Cameron-Martin Formula, see e.g. \cite[Appendix B.1]{ustunelzakai}$)$ Let $F$ be a $\mathcal F_T$-measurable random variable and let $h$ be in $H$. Then, when both sides are well-defined
$$ \E[F\tr{h}] =\E\left[ F \exp\left(\int_0^T \dot{h}(s) \cdot dW_s -\frac12 \int_0^T \|\dot{h}(s)\|^2 ds\right)\right]. $$
\end{Proposition}

For further reference, we also emphasize that for any $h\in H$ and for any $p\geq 1$, the stochastic exponential $\mathcal E\left(\int_0^\cdot\dot{h}(s)\cdot dW_s\right):=\exp\left(\int_0^\cdot \dot{h}(s) dW_s -\frac12 \int_0^\cdot \|\dot{h}(s)\|^2 ds\right)$ verifies
\begin{equation}\label{exp.int}
\mathcal E\left(\int_0^\cdot\dot{h}(s)\cdot dW_s\right)\in\mathbb S^p, \quad \forall p\geq 1.
\end{equation}

\begin{Lemma}
\label{lemma:shift1}
Let $t$ in $[0,T]$ and let $F$ be a $\mathcal{F}_t$-measurable random variable. For any $h$ in $H$, it holds that 
$$ F\tr{h}= F\circ {\tau_{\widetilde{h^t}}}, \ \P_0-a.s,$$
where $$\widetilde{h^t}(s):=\left(\int_0^s \dot{ h}^1(u) \mathbf 1_{0\leq u\leq t}du, \dots, \int_0^s \dot{ h}^d(u) \mathbf 1_{0\leq u\leq t}du\right)^\top.$$
In particular, $F\tr{h}$ is $\mathcal{F}_t$-measurable. 
\end{Lemma}

\begin{proof}
It is well-known that by definition of $\P_0$, any $\Fc_t$-measurable random variable admits a $\Fc_t^o$-measurable version. Therefore, there exists some measurable map $\varphi:\Omega \to \real$, such that 
$$F=\varphi(W_{\cdot\wedge t}),\ \mathbb P_0-a.s.$$
Hence, we deduce by Lemma \ref{lemma:UstunelZakai} that for $\P_0-a.e.$ $\omega\in\Omega$
$$F\tr{h}(\omega)=\varphi(W_{\cdot\wedge t}(\omega))\tr{h}=\varphi(W_{\cdot\wedge t}\tr{h}(\omega))=\varphi(\omega(\cdot\wedge t)+h(\cdot \wedge t))=F\circ \tau_{\widetilde{h^t}}(\omega).$$
\qed
\end{proof}

We conclude this section with the following lemma which might be known. However since we did not find it in the literature we provide a proof in order to make this paper self-contained.

\begin{Lemma}
\label{lemma:shift2}
Let $Z\in\H_{d}^2$ and $h$ in $H$. It holds that
$$ \int_0^T Z_s \cdot dW_s \tr{h} = \int_0^T Z_s\tr{h} \cdot dW_s+\int_0^T Z_s \tr{h} \cdot \dot{h}(s)ds, \ \P_0-a.s.$$
\end{Lemma}

\begin{proof}
Let $\mathfrak{S}$ be the class of simple processes $X$ of the form
$$X_t:=(X_t^1,\ldots,X_t^d)^\top, \quad X_t^j=\sum_{i=0}^{n_j} \lambda_{i}^j\mathbf{1}_{(t_i^j,t_{i+1}^j]}(t),\; j\in \{1,\ldots,d\}$$
where for any $j\in \{ 1, \ldots,d\}$, $n_j\in\N^*$, $t_0^j=0<t_1^j<...<t_n^j=T$ and where for any $0\leq i\leq n_j$, $(\lambda_i^j)_{i=1,...n_j}$ are $\mathcal{F}_{t_i^j}$-measurable and in $L^2(\R)$.

\vspace{0.5em}
We start by proving the result for $Z$ in $\mathfrak{S}$ and then we prove the result for any element $Z$ in $\H_{ d}^2$ using a density argument.   
Let $Z\in \mathfrak{S}$ with the decomposition
$$Z_s= (Z_s^1,\ldots,Z_s^d)^\top, \quad Z_s^j:=\sum_{i=0}^{n_j} \lambda_i^j\mathbf{1}_{(t_i^j,t_{i+1}^j]}(s), \ s\in[0,T], \ j\in \{1,\ldots, d \}.$$ Then, for any $h\in H$ and for every $\omega\in \Omega$,
\begin{align*} 
\left(\int_0^T Z_s\cdot dW_s\circ \tau_h\right)(\omega)&= \left(\sum_{j=1}^d \sum_{i=0}^{n_j} \lambda_{i}^j(W_{t_{i+1}^j}^j-W_{t_i^j}^j)\right) \tr{h} (\omega)\\
&= \sum_{j=1}^d \sum_{i=0}^{n_j} \lambda_{i}^j(\omega+h)(W_{t_{i+1}^j}^j-W_{t_i^j}^j)(\omega+h)\\
& =\sum_{j=1}^d \sum_{i=0}^{n_j} \lambda_i^j \tr{h} (\omega) \left(\omega^j(t_{i+1}^j)-\omega(t_i^j)+h^j(t_{i+1}^j)-h^j(t_i^j)\right)\\
&=\int_0^T Z_s\circ \tau_h \cdot dW_s(\omega)+\int_0^T Z_s\tr{h}(\omega)\cdot dh_s,
\end{align*}
which gives the desired result since $h$ is absolutely continuous. We extend this result to processes $Z$ in $\H_{ d}^2$. Let $Z\in \H_{ d}^2$, then there exists a sequence $(Z^n)_{n\in \N}$ in $\mathfrak{S}$ which converges to $Z$ in $\H_{ d}^2$. Hence,
\begin{align*}
&\E\left[ \abs{\int_0^T Z_s \cdot dW_s \circ \tau_h-\int_0^T Z_s\circ \tau_h \cdot dW_s-\int_0^T Z_s\tr{h}\cdot dh_s}\right]\\
&\leq \E\left[ \abs{\int_0^T Z_s \cdot dW_s \circ \tau_h-\int_0^T Z^n_s \cdot dW_s \circ \tau_h}\right]+\E\left[ \abs{\int_0^T Z^n_s \circ \tau_h \cdot dW_s -\int_0^T Z_s \circ \tau_h \cdot dW_s}\right]\\
&\hspace{0.9em}+\E\left[ \abs{\int_0^T Z^n_s \circ \tau_h\cdot dh_s -\int_0^T Z_s \circ \tau_h \cdot dh_s}\right]\\
&\leq  \underbrace{\E\left[ \abs{\int_0^T (Z_s - Z^n_s) \cdot dW_s  }\circ \tau_h\right]}_{=:A^n}+\underbrace{\E\left[ \abs{\int_0^T (Z^n_s - Z_s)\circ \tau_h \cdot dW_s}\right]}_{=:B^n}\\
&\hspace{0.9em}+\underbrace{\E\left[ \abs{\int_0^T (Z^n_s - Z_s)\circ \tau_h \cdot dh_s}\right]}_{=:C^n}.
\end{align*}

Let us estimate these three terms. First, using Proposition \ref{prop.cam}, Cauchy-Schwarz Inequality, then Burkholder-Davis-Gundy Inequality, we have
\begin{align*}
A^n= \E\left[ \abs{\int_0^T(Z_s-Z_s^n) \cdot dW_s} e^{\int_0^T\dot{h}(s) \cdot dW_s-\frac12 \int_0^T \|\dot{h}(s)\|^2ds}\right]\leq &\ \E\left[\int_0^T \|Z_s-Z_s^n\|^2 ds\right]^{1/2}\\
&\times\E\left[\mathcal E\left(\int_0^\cdot\dot{h}(s) \cdot dW_s\right)_T^2\right]^{1/2}.
\end{align*}
By \reff{exp.int}, this clearly goes to $0$ as $n$ goes to infinity. Similarly, using Burkholder-Davis-Gundy Inequality, we have
$$ B^n\leq  \E\left[\left(\int_0^T \left\|(Z^n_s - Z_s)\circ \tau_h\right\|^2 ds\right)^\frac12\right]= \E\left[\left(\int_0^T \|Z^n_s - Z_s\|^2 ds\right)^\frac12 \circ \tau_h\right].$$ 
Therefore, we can use Proposition \ref{prop.cam} and Cauchy-Schwarz Inequality, to also deduce that $B^n\underset{n\to +\infty}{\to} 0$. Finally, we have
\begin{align*}
C^n=&\ \E\left[\mathcal E\left(\int_0^T\dot{h}(s) \cdot dW_s\right)\abs{\int_0^T (Z_s^n-Z_s) \cdot \dot{h}(s)ds}\right]\\
\leq&\ \E\left[\mathcal E\left(\int_0^T\dot{h}(s)\cdot dW_s\right)^2\right]^{1/2}\E\left[\left(\int_0^T\|Z_s^n-Z_s\| \|\dot{h}(s)\| ds\right)^2\right]^{1/2}\\
\leq &\ \E\left[\mathcal E\left(\int_0^T\dot{h}(s) \cdot dW_s\right)^2\right]^{1/2}\E\left[\int_0^T\|Z_s^n-Z_s\|^2ds\right]^{1/2}\left(\int_0^T\|\dot{h}(s)\|^2ds\right)^{1/2},
\end{align*}
which also goes to $0$ as $n$ goes to infinity. Therefore the proof is complete.
\qed
\end{proof}

\vspace{0.5em}
This result entails the following useful consequence. Let $t$ in $(0,T]$ and $h$ in $H$ such that $\dot{h}_s=(0,\ldots,0)$ for $s\geq t$. Then for any $Z$ in $\H_{ d}^2$, it holds that:
\begin{equation}
\label{eq:intshiftprev}
\int_t^T Z_s \cdot dW_s \tr{h} = \int_t^T Z_s \tr{h}\cdot  dW_s, \quad \P_0-a.s., 
\end{equation}
since $\int_t^T Z_s \tr{h} \cdot \dot{h}(s) ds=0$.

\section{A characterization of Malliavin differentiability}
\label{section:Malliavin}

Before going further, we would like to recall the main finding of \cite{Sugita}. Any Malliavin-Sobolev type space $\D^{1,p}$ as defined in Section \ref{sec:preli} (originally defined by Malliavin \cite{Malliavin} and Shigekawa \cite{Shige}) agrees with the Sobolev space (due to Stroock \cite{stroock} and Kusuoka \cite{Kusuoka}) $\tilde{\D}^{1,p}$ consisting in the set of  \textit{Ray Absolutely Continuous} (RAC) and \textit{Stochastically G\^ateaux Differentiable} (SGD) r.v. $F$ in $L^p(\real)$, where these notions are defined as follows:
\begin{itemize}
\item[(RAC)] For any $h$ in $H$, there exists a r.v. $\tilde{F}_h$ such that $\tilde{F}_h=F$, $\P_0-$a.s., and such that for any $\omega$ in $\Omega$, $t \in \real\longmapsto \tilde{F}_h(\omega+t h)$ is absolutely continuous, where $t h:=(t h^1,\ldots,t h^d)$.
\item[(SGD)] There exists $\mathcal{D}F$ in $L^p(H)$ such that for any $h$ in $H$,
\begin{equation}
\label{eq:cvproba}
\frac{F\tr{\varepsilon h}-F}{\varepsilon} \underset{\varepsilon \to 0}{\longrightarrow} \langle \mathcal{D}F, h\rangle_H,\ \text{in probability.}
\end{equation}
\end{itemize} 
In addition, for any $F$ in $\D^{1,p}$, $\nabla F=\mathcal{D}F$, $\P_0-$a.s.
Note that according to the statement of Step 1 in the proof of \cite[Theorem 3.1]{Sugita}, if $F$ is (RAC) and (SGD) then for any $h$ in $H$ and any $\varepsilon>0$ it holds that
$$\varepsilon^{-1}(\tilde{F}_h\tr{\varepsilon h} -\tilde{F}_h) =\varepsilon^{-1} \int_0^\varepsilon \langle \nabla F \tr{s h}, h\rangle_H ds, \ \P_0-a.s.$$
Furthermore, by Lemma \ref{lemma:UstunelZakai}, we have for any $\varepsilon$ there exists a set $A^\varepsilon$ such that $\P_0[A^\varepsilon]=0$ and $F\tr{\varepsilon h}=\tilde{F}_h \tr{\varepsilon h}$ and $F=\tilde F_h$ outside $A^\varepsilon$. Hence, for any $\varepsilon$ in $(0,1)$, the relation above rewrites as:
\begin{equation}
\label{eq:cons}
\varepsilon^{-1} (F\tr{\varepsilon h}-F) = \varepsilon^{-1} \int_0^\varepsilon \langle \nabla F \tr{s h}, h\rangle_H ds, \ \P_0-a.s.
\end{equation}  
\begin{Remark}
It has actually been proved by Janson \cite{janson} that \eqref{eq:cons} is equivalent to $($RAC$)$ and $($SGD$)$, for any $p>1$, see Lemma 15.89. Notice that \cite{janson} also obtained a similar characterization for $p=1$ $($see Lemma 15.71$)$. However, as stated in Remark 4 of \cite{Sugita}, the identification of the Kusuoka-Stroock and Shigekawa spaces when $p=1$ is still an open result, so that we never consider the case $p=1$ in this paper.
\end{Remark}

The main result of this section is the following theorem whose proof is postponed to the end of the section.
\begin{Theorem}
\label{th:newcarD}
Let $p>1$ and $F\in L^p(\real)$. The following properties are equivalent
\begin{itemize}
\item[$($i$)$] $F$ belongs to $\D^{1,p}$.
\item[$($ii$)$] There exists $\mathcal{D}F$ in $L^p(H)$ such that for any $h$ in $H$ and any $q\in [1,p)$
$$ \lim\limits_{\varepsilon\to 0} \E\left[ \abs{ \frac{F\circ \tau_{\varepsilon h}-F}{\varepsilon}-\langle \mathcal{D}F, h\rangle_H}^q\right]=0. $$
\item[$($iii$)$] There exists $\mathcal{D}F$ in $L^p(H)$ and there exists $q\in [1,p)$ such that for any $h$ in $H$
$$ \lim\limits_{\varepsilon\to 0} \E\left[ \abs{ \frac{F\circ \tau_{\varepsilon h}-F}{\varepsilon}-\langle \mathcal{D}F, h\rangle_H}^q\right]=0. $$
\item[$($iv$)$] There exists $\mathcal{D}F$ in $L^p(H)$ such that for any $h$ in $H$
$$ \lim\limits_{\varepsilon\to 0} \E\left[ \abs{ \frac{F\circ \tau_{\varepsilon h}-F}{\varepsilon}-\langle \mathcal{D}F, h\rangle_H}\right]=0. $$
\end{itemize}
In that case, $\mathcal D F= \nabla F$.
\end{Theorem}

\begin{Remark}
The implication $(ii) \Rightarrow (i)$ when $q=p=2$ already appears in \cite{bog} $($see 8.11.3$)$. This is of course contained in our result.
\end{Remark}
We now give the following lemma which characterizes the Malliavin derivative using the duality formula involving the Skorohod operator (also called divergence operator). 

\begin{Lemma}
\label{lemma:Sugita}
Let $\varepsilon>0$ and $1<p<+\infty$. Suppose that $F\in L^{1+\varepsilon}(\R)$ and assume that there exists $\mathcal{D}F$ in $L^p(H)$ such that:
$$ \E\left[ F\delta(G h)\right]=\E\left[ G \langle \mathcal{D}F,h\rangle_H\right],$$
for every $G\in \mathcal S$ and $h\in H$. Then, it holds that $F\in \D^{1,p}$, and $\mathcal{D}F=\nabla F, \; \P_0-a.s.$
\end{Lemma}
\begin{proof}
We know (see e.g. \cite[Corollary 2.1]{Sugita}) that the result is true if $\mathcal{S}$ is replaced by $\mathcal{P}$. Let $G$ be in $\mathcal{P}$. By Lemma \ref{lemma:tec1} there exists $(G^N)$ in $\mathcal{S}$ such that $G^N$ approximates $G$ in $\mathbb{D}^{1,p}$. Let $h$ in $H$. For any $N\geq 1$, we have 
\begin{align*}
\E[F \delta(Gh)] &= \E[F (G W(h)-\langle \nabla G,h\rangle_H)]\\ 
&= \E[F (G^N W(h)-\langle \nabla G^N,h\rangle_H)] - \E[(G^N-G) F W(h)- F \langle \nabla (G^N-G),h\rangle_H]\\
&= \E[F \delta(G^Nh)] - \E[(G^N-G) F W(h)- F \langle \nabla (G^N-G),h\rangle_H]\\
&= \E[G^N \langle \mathcal{D} F, h\rangle_H] - \E[(G^N-G) F W(h)- F \langle \nabla (G^N-G),h\rangle_H].
\end{align*}
Furthermore, by Lemma \ref{lemma:tec1}
$$ |\E[(G^N-G) F W(h)]| \leq \E[|F W(h)|^p]^{1/p} \E[|G^N-G|^{\bar{p}}]^{1/\bar{p}} \underset{N\to +\infty}{\longrightarrow} 0,$$
with $1<p<1+\varepsilon$ and where $\bar p$ is the conjugate of $p$, and
$$ |\E[F \langle \nabla (G^N-G),h\rangle_H]|\ \leq \E[|F|^p]^{1/p} \E[\|\nabla (G^N-G)\|_H^{\bar p}]^{1/\bar p} \|h\|_H \underset{N\to +\infty}{\longrightarrow} 0,$$
by Lemma \ref{lemma:tec1} again. Hence,
\begin{align*}
\E[F\delta(Gh)] &= \lim_{N\to+\infty}  \E[G^N \langle \mathcal{D} F, h\rangle_H]\\
&=\E[G \langle \mathcal{D} F, h\rangle_H] + \lim_{N\to+\infty}  \E[(G^N-G) \langle \mathcal{D} F, h\rangle_H],
\end{align*}
and
$$ \lim_{N\to+\infty} |\E[(G^N-G) \langle \mathcal{D} F, h\rangle_H]| \leq \lim_{N\to+\infty} \E[|G^N-G|^p]^{1/p} \E[\|\mathcal{D}F\|_H^{\bar p}]^{1/\bar p} \|h\|_H \underset{N\to +\infty}{\longrightarrow} 0.$$
Thus we have proved that for any $G$ in $\mathcal{P}$ and for any $h$ in $H$,
$$\E[F\delta(Gh)] = \E[G \langle \mathcal{D} F, h\rangle_H],$$
which gives the result by \cite[Corollary 2.1]{Sugita}.
\qed
\end{proof}

We now prove the following lemma for the Malliavin differentiability of a given random variable.

\begin{Lemma}
\label{lemma:D12nec}
Let $p>1$. Let $F$ be in $\D^{1,p}$. Then, for any $q$ in $[1,p)$ and for any $h$ in $H$, 
$$ \frac{F\tr{\varepsilon h}-F}{\varepsilon}\underset{\varepsilon \to 0}{\longrightarrow} \langle \nabla F, h\rangle_H \textrm{ in } L^q(\real).$$
\end{Lemma}

\begin{proof}
Fix $q$ in $[1,p)$, $h$ in $H$ and $\eta>0$ such that $q+\eta<p$. We know from \cite[Theorem 3.1]{Sugita} that since $F$ is in $\D^{1,p}$, $F$ is (SGD), 
(RAC), and Relation \eqref{eq:cons} holds true.
We thus have using Jensen's Inequality
\begin{align*}
\E\left[\left|\varepsilon^{-1} (F \tr{\varepsilon h}-F) \right|^{q+\eta}\right]&=\E\left[\varepsilon^{-(q+\eta)} \left| \int_0^\varepsilon \langle \nabla F \tr{s h}, h\rangle_H ds \right|^{q+\eta}\right]\\
&\leq\varepsilon^{-1} \E\left[\int_0^\varepsilon \left|\langle \nabla F \tr{s h}, h\rangle_H \right|^{q+\eta} ds\right]\\
&=\varepsilon^{-1} \int_0^\varepsilon \E\left[\left|\langle \nabla F, h\rangle_H \right|^{q+\eta} \tr{s h}\right] ds\\
&=\varepsilon^{-1} \int_0^\varepsilon \E\left[\left|\langle \nabla F, h\rangle_H \right|^{q+\eta} \mathcal{E}\left(s \int_0^T \dot{h}_r \cdot dW_r\right) \right]ds\\
&\leq \E\left[\left|\langle \nabla F, h\rangle_H \right|^p\right]^{\frac{q+\eta}{p}} \sup_{t\in (0,1)} \E\left[\left|\mathcal{E}\left(t \int_0^T \dot{h}_r \cdot dW_r\right)\right|^{\frac{p}{p-q-\eta}} \right]^{\frac{p-q-\eta}{p}}\\
&<+\infty.
\end{align*}
Hence by de La Vall\'ee Poussin Criterion, we deduce that the family of random variables $\left(\left|\varepsilon^{-1}(F\tr{\varepsilon h} -F)\right|^q\right)_{\varepsilon \in (0,1)}$ is uniformly integrable
which together with the convergence in probability \eqref{eq:cvproba} gives the result. 
\qed
\end{proof}

\begin{Remark}
Note that the conclusion of the previous Lemma may fail for $q=p$\footnote{After the first version of this paper, a counter example has been given in \cite{IMPR}.}
\end{Remark}

We can now proceed with the proof of Theorem \ref{th:newcarD}.

\vspace{0.5em}
\textbf{Proof of Theorem \ref{th:newcarD}.} From Lemma \ref{lemma:D12nec} we have $(i)\Rightarrow (ii)$ and of course $(ii) \Rightarrow (iii)\Rightarrow (iv)$. We turn to $(iv)\Rightarrow (i)$. Let $F$ be such that there exists $\mathcal{D}F$ in $L^p(H)$ such that
$$ \lim\limits_{\varepsilon \to 0} \E\left[ \abs{ \frac{F\circ \tau_{\varepsilon h}-F}{\varepsilon}-\langle \mathcal{D}F, h\rangle_H}\right]=0. $$
The proof consists in applying Lemma \ref{lemma:Sugita} by proving the duality relationship
\begin{equation}
\label{eq:proofdual}
\E[F \delta(G h)] = \E\left[G \langle \mathcal{D}F, h\rangle_H \right], \ G\in \mathcal S, \;h \in H.
\end{equation}
By Lemma \ref{lemma:SugitabisStep1Step1} (in the Appendix) with $\varepsilon=0$,
\begin{align}
\label{eq:temp1}
\E\left[ F \delta(G h) \right]= &\ \frac{d}{d\varepsilon }\E\left[ F\tr{\varepsilon h} \, G\right]_{\vert \varepsilon =0}\nonumber\\
=&\ \lim\limits_{\eta \to 0}\frac1\eta \E\left[(F\tr{\eta h}-F) G \right]\nonumber\\
=&\ \lim\limits_{\eta \to 0}\E\left[\left(\frac{F\circ \tau_{\eta h}-F}{\eta}- \langle \mathcal{D}F,h \rangle_H \right) G \right]\nonumber \\
&+\mathbb{E}\left[ \langle \mathcal{D}F,h \rangle_H G\right]\nonumber\\
=&\ \mathbb{E}\left[ \langle \mathcal{D}F,h \rangle_H G\right],
\end{align}
where the proof that the first term on the right-hand side goes to $0$ is reported below.

\vspace{0.5em}
Note that $\E[|\langle \mathcal{D}F,h \rangle_H G|]<+\infty$ since $G$ is bounded and $\mathcal{D}F$ belongs to $L^p(H)$.
The Equality \eqref{eq:temp1} is justified by H\"older's inequality  since
\begin{align*}
&\E\left[\abs{\left(\frac{F\circ \tau_{\eta h}-F}{\eta}-\langle \mathcal{D}F,h \rangle_H\right) G} \right]\\
&\leq \|G\|_\infty \E\left[\abs{\frac{F\circ \tau_{\eta h}-F}{\eta}-\langle \mathcal{D}F,h \rangle_H} \right]\underset{\varepsilon \to 0}{\longrightarrow} 0.
\end{align*}
\qed

\begin{Corollary}
Let $F$ be in $\mathbb D^{1,p}$. For any $\varepsilon>0$ and any $h$ in $H$, $F\tr{\varepsilon h}$ belongs to $\D^{1,p}$ and $\nabla (F\tr{\varepsilon h})=(\nabla F)\tr{\varepsilon h}$. 
\end{Corollary}
\begin{proof}
Let $F$ be in $\mathbb D^{1,p}$. Using Theorem \ref{th:newcarD}, we know that for any $h$ in $H$ and any $q\in [1,p)$
$$ \lim\limits_{\varepsilon\to 0} \E\left[ \abs{ \frac{F\circ \tau_{\varepsilon h}-F}{\varepsilon}-\langle \nabla F, h\rangle_H}^q\right]=0. $$
By Lemma \ref{lemma:SugitabisStep1Step1} (in the Appendix) it holds that
\begin{align}
\label{eq:temp1}
\E\left[ F\tr{\varepsilon h} \; \delta(G h) \right]= &\ \frac{d}{d\varepsilon }\E\left[ F\tr{\varepsilon h} \, G\right]\nonumber\\
=&\ \lim\limits_{\eta \to 0}\frac1\eta \E\left[(F\tr{(\varepsilon +\eta)h}-F\circ \tau_{\varepsilon h}) G \right]\nonumber\\
=&\ \lim\limits_{\eta \to 0}\E\left[\left(\frac{F\circ \tau_{(\varepsilon +\eta)h}-F\circ \tau_{\varepsilon h}}{\eta}- \langle (\nabla F)\circ \tau_{\varepsilon h},h \rangle_H \right) G \right]\nonumber \\
&+\mathbb{E}\left[ \langle (\nabla F)\circ \tau_{\varepsilon h},h \rangle_H G\right]\nonumber\\
=&\ \mathbb{E}\left[ \langle (\nabla F)\circ \tau_{\varepsilon h},h \rangle_H G\right],
\end{align}
where the proof that the first term on the right-hand side goes to $0$ is reported below.

\vspace{0.5em}
Note that $\E[|\langle (\nabla F)\circ \tau_{\varepsilon h},h \rangle_H G|]<+\infty$ since $\langle (\nabla F)\circ \tau_{\varepsilon h},h \rangle_H=\langle \nabla F,h \rangle_H \circ \tau_{\varepsilon h}$, $\P_0-$a.s., $G$ belongs to all the spaces $L^r(\R)$ for $r\geq 1$ and 
$$ \E[|\langle \nabla F,h \rangle_H|^p ] \leq \|h\|_H^{p} \E\left[\|\nabla F\|_H^p\right]<+\infty.$$
The Equality \eqref{eq:temp1} is justified by H\"older's inequality  since
\begin{align*}
&\E\left[\abs{\left(\frac{F\circ \tau_{(\varepsilon +\eta)h}-F\circ \tau_{\varepsilon h}}{\eta}-\langle \nabla F\circ \tau_{\varepsilon h},h \rangle_H\right) G} \right]\\
&\leq \E\left[\abs{\frac{F\circ \tau_{(\varepsilon +\eta)h}-F\circ \tau_{\varepsilon h}}{\eta}-\langle \nabla F\circ \tau_{\varepsilon h},h \rangle_H}^r\right]^{\frac1r} \E[|G|^{\bar{r}}]^{\frac{1}{\bar{r}}}\\
&= \E\left[\left|\frac{F\circ \tau_{\eta h}-F}{\eta}-\langle \nabla F,h \rangle_H \right|^r \circ \tau_{\varepsilon h} \right]^{\frac1r} \E[|G|^{\bar{r}}]^{\frac{1}{\bar{r}}}\\
&\leq \E\left[\left|\frac{F\circ \tau_{\eta h}-F}{\eta}-\langle \nabla F,h \rangle_H \right|^q\right]^{\frac1q} \E\left[\mathcal{E}\left(\varepsilon  \int_0^T \dot{h}(s)\cdot  dW_s\right)^{\bar{\alpha}} \right]^{\frac{1}{r\bar{\alpha}}} \E[|G|^{\bar{r}}]^{\frac{1}{\bar{r}}}
\end{align*}
where $1<r<q$ and $\alpha:=\frac{q}{r}$ and where $\bar{r}$ (resp. $\bar{\alpha}$) is the H\"older conjugate of $r$ (resp. $\alpha$). Consequently, $\E\left[ F\tr{\varepsilon h} \; \delta(G h) \right]=\mathbb{E}\left[ \langle \nabla F\circ \tau_{\varepsilon h},h \rangle_H G\right]$, and from Lemma \ref{lemma:Sugita}  $ \nabla (F\circ \tau_{\varepsilon h})=(\nabla F)\tr{\varepsilon h}$.
\qed
\end{proof}
\section{Malliavin's differentiability of BSDEs}
\label{section:BSDEs}

In this section we derive a sufficient condition ensuring that the solution to a BSDE is Malliavin differentiable. To simplify the comparison of the results with the companion papers \cite{EPQ,pardouxpeng} we adopt the notations used in these papers concerning the Malliavin calculus. More precisely, for any $F$ in $\D^{1,p}$ (for $p>1$) we have defined the Malliavin derivative $\nabla F$ as an $H$-valued random variable. Recall that denoting $DF$ the derivative of $\nabla F$ that is $\nabla_t F=\int_0^t D_r F dr$, $DF$ coincides with the Malliavin derivative introduced in \cite{EPQ,pardouxpeng,Nualartbook}. In particular $\langle \nabla F, h\rangle_H=\langle DF,\dot{h} \rangle_\h$ for any $h$ in $H$.
 
\vspace{0.5em} Let $n$ be a positive integer, we consider now the following BSDE:
\begin{equation}
Y_t=\xi+ \int_t^T f(r,Y_r,Z_r)dr-\int_t^T Z_r  dW_r, \ t\in [0,T], \ \P_0-a.s., \label{bsde}
\end{equation}
where $\xi:=(\xi^1,\cdots,\xi^n)^{\top}$ is a $\mathcal{F}_T$-measurable $\R^{n}$-valued r.v. and $f:[0,T]\times\Omega\times\real^n\times\real^{n\times d} \longrightarrow \real^n$ is a $\F$-progressively measurable process where as usual the $\omega$-dependence is omitted.
 
\vspace{0.5em}
The aim of this section is to show that for any $t\in[0,T]$, we can apply Theorem \ref{th:newcarD} under the following assumptions:

\begin{itemize}
\item[(L)] The map $(y,z)\longmapsto f(\cdot,y,z)$ is differentiable with uniformly bounded and continuous partial derivatives. We denote by $f_y:=\left(\frac{\partial f^j}{\partial y_k}\right)_{j\in\{1,\ldots,n\}, \ k\in\{1,\ldots, n \}}$ the Jacobian matrix of $f$ with respect to $y$, where $j$ (resp. $k$) indexes the columns (resp. the rows) of $f_y$ and $f^j_y$ denotes the gradient of $f^j$. We denote by $f_z^j$, for any $j\in \{1,\ldots,n \}$ the Jacobian matrix of $f^j$ with respect to $z$, that is $f_z^j=\left( \frac{\partial f^j}{\partial z_{k,l}}\right)_{k\in \{1,\ldots, n \}, \ l \in \{1,\ldots, d\}}$.
\item[(D)] $\xi$ belongs to $(\D^{1,2})^n$, for any $(y,z)\in\R^n\times\real^{n\times d}$, the map $(t,\omega)\longmapsto f(t,\omega,y,z)$ is in $L^2([0,T];(\D^{1,2})^n)$, $f(\cdot,y,z)$ and $Df(\cdot,y,z)$ are $\F$-progressively measurable, and
$$ \E\left[\int_0^T\sum_{j=1}^n \No{D_\cdot f^j(s,Y_s,Z_s)}_\h^2 ds\right]<+\infty. $$
\item[($H_1$)] There exists $p\in(1,2)$ such that for any $h\in H$ and for any $j \in\{1,\ldots,n\}$
$$\lim\limits_{\varepsilon \to 0}\ \E\left[ \left(\int_0^T\left|\frac{f^{ j}(s,\cdot+\varepsilon h, Y_s, Z_s)-f^{ j}(s,\cdot,Y_s,Z_s)}{\varepsilon}-\langle Df^{ j}(s,\cdot,Y_s,Z_s),\dh\rangle_{\h}\right| ds\right)^p\right]=0,$$ 
\item[($H_2$)] Let $(\varepsilon_k)_{k\in \N}$ be a sequence in $(0,1]$ such that $\lim\limits_{k\to +\infty} \varepsilon_k=0$, and let  $(Y^k,Z^k)_k$ be a sequence of random variables which converges in $\S_{ n}^p\times \H_{n,d}^p$ for any $p\in[1,2)$ to some $(Y,Z)$. Then for all $h\in H$ and for all $j\in\{1,\ldots,n\}$, the following convergences hold in probability
\begin{align}
\label{assumpcts1}
&\No{f^j_y(\cdot,\omega+\varepsilon_k h,Y^k_\cdot,Z_\cdot) -f^j_y(\cdot,\omega,Y_\cdot,Z_\cdot)}_{L^2([0,T];\mathbb R^{\b n})}\underset{k\to+\infty}{\longrightarrow}0 \nonumber\\
& \No{f^j_z(\cdot,\omega+\varepsilon_k h,Y^k_\cdot,Z^k_\cdot)-  f^j_z(\cdot,\omega,Y_\cdot,Z_\cdot)}_{L^2([0,T];\mathbb R^{n\times d})}\underset{k\to+\infty}{\longrightarrow}0,\end{align}
or
\begin{align}
\label{assumpcts2}
&\No{f_y^j(\cdot,\omega+\varepsilon_k h,Y^k_\cdot,Z^k_\cdot) -f_y^j(\cdot,\omega,Y_\cdot,Z_\cdot)}_{L^2([0,T];\mathbb R^{\b n})}\underset{k\to +\infty}{\longrightarrow}0\nonumber\\
& \No{f_z^j(\cdot,\omega+\varepsilon_k h,Y_\cdot,Z^k_\cdot)   -f_z^j(\cdot,\omega,Y_\cdot,Z_\cdot)}_{L^2([0,T];\mathbb R^{n\times d})}\underset{k\to +\infty}{\longrightarrow}0.
\end{align}
\end{itemize}

Before turning to the main result of this section, we would like to comment on Assumption $(H_2)$. In this explanation we set $n=1$ for the sake of simplicity. On the one hand, by Lemma \ref{lemma:ctsH}, at given $(s,y,z)$, $f_y(s,\omega+\varepsilon_k h,y,z)$ converges in probability to $f_y(s,\omega,y,z)$ as $n$ goes to infinity. On the other hand, $f_y(s,\omega,\cdot)$ is continuous by assumption. Thus, Condition $(H_2)$ is just requiring joint continuity of $f_y$ in $L^2([0,T])$. The same comment holds for $f_z$. Note finally, that since $f_y$ is assumed to be bounded, a sufficient condition for $(H_2)$ to hold true is that $f_y(t,Y_t^k,Z_t)$ converges in probability to $f_y(t,Y_t,Z_t)$ for $dt$-almost every $t$ (and the same for $f_z$). 

\begin{Theorem}\label{thm:yzd12}
Let $t$ be in $[0,T]$. Under Assumptions $(L)$, $(D)$, $(H_1)$ and $(H_2)$, $Y_t$ belongs to $(\D^{1,2})^n$ and $(Z^{j,:})^\top \in L^2([t,T];(\D^{1,2})^{d}), \; j\in \{1,\ldots,n\}$.
\end{Theorem}

\begin{proof}
We only consider the case where \reff{assumpcts1} holds in Assumption $(H_2)$, since the other one can be treated similarly. We prove first that $Y_t^{j}$ belongs to $\D^{1,p}$ for any $j$ in $\{1,\ldots,n\}$ where $p\in(1,2)$ is the exponent appearing in Assumption $(H_1)$, and then we extend the result to $\D^{1,2}$. To this end we aim at applying Theorem \ref{th:newcarD}. Fix $j$ in $\{1,\ldots,n\}$. Let $h$ in $H$. Since $Y^{ j}$ is $\F$-progressively measurable, by Lemma \ref{lemma:shift1}, we can assume without loss of generality that $\dot{h}_s=0$ for $s>t$. Let $\varepsilon>0$. By Lemmas \ref{lemma:UstunelZakai} and \ref{lemma:shift2}, it holds that
$$ Y_s^{j}\tr{\varepsilon h} =\xi^{j}\tr{\varepsilon h} +\int_s^T f^{j}(r,Y_r,Z_r) \tr{\varepsilon h} dr -\int_s^T (Z_r^{j,:})^\top\tr{\varepsilon h} \cdot dW_r, \quad \forall s \in [t,T], \; \P_0-a.s.  $$
As a consequence, setting for the sake of simplicity
$$Y_s^\varepsilon:=\frac1\varepsilon (Y_s\circ \tau_{\varepsilon h}-Y_s),\quad Z_s^\varepsilon:=\frac1\varepsilon (Z_s\circ \tau_{\varepsilon h}-Z_s), \quad \xi^{\varepsilon}:=\frac1\varepsilon (\xi\circ \tau_{\varepsilon h}-\xi), \; s \in [t,T],$$ 
we have for any $j\in \{1,\ldots, n \}$
\begin{align}
\nonumber (Y_s^{\varepsilon})^{j}=&\ (\xi^{\varepsilon})^{j}+ \int_s^T \left(\tilde{A}_r^\varepsilon+\tilde{A}_r^{y,\varepsilon}\cdot Y_r^\varepsilon+\sum_{k=1}^d (f_z^j)^{:,k}(r,\cdot+\varepsilon h,Y_r \tr{\varepsilon h}, \widetilde{Z}^k_r)\cdot (Z_r^{\varepsilon})^{:,k} \right)dr\\
&\ \label{edsr_accroissement}-\int_s^T( (Z_r^{\varepsilon})^{j,:} )^\top \cdot dW_r,
\end{align}
with
\begin{align*}
    &\tilde{A}_r^{y,\varepsilon}:=f_y^{j}(r,\cdot+\varepsilon h, \bar{Y}_r^{\varepsilon,h}, Z_r),\ \tilde{A}_r^\varepsilon:=\frac1\varepsilon (f^{j}(r,\cdot+\varepsilon h, Y_r, Z_r)-f^{j}(r,\cdot,Y_r,Z_r)),
 \end{align*}
where $\bar{Y}_r^{\varepsilon,h}$ is a convex combination of $Y_r$ and $Y_r\circ\tau_{\varepsilon h}$ and where for any $k\in \{1,\ldots,d \}$, we have  $\widetilde Z^k_r:=(Z_r^{:,1}\tr{\varepsilon h},\ldots,Z_r^{:,k-1} \tr{\varepsilon h}, \bar{Z}_r^{:,k}, Z_r^{:,k+1}, \ldots , Z_r^{:,d})$ where $\bar{Z}^{:,k}_r$ is a convex combination of $Z_r^{:,k} \tr{\varepsilon h}$ and $Z_r^{:,k}.$ \vspace{0.5em}

Under Assumptions (D) and (L), the following linear BSDE on $[t,T]$ has a unique solution $(\hat{Y}^{h}, \hat{Z}^h)$ in $\mathbb S^2_n\times \mathbb H^2_{n,d}$ with $\hat{Y}^h:=( (\hat{Y}^{h})^{j})_{j\in \{1,\ldots,n \}}$, $ \hat{Z}^h:=( (\hat{Z}^h)^{j,:})_{j\in \{1,\ldots,n \}}$ and for any $j\in \{1,\ldots, n \}$
\begin{align}
\nonumber
(\hat{Y}^{h}_s)^{j}=&\ \langle \nabla(\xi^j),h\rangle_{H}+\int_s^T\Big( \langle \nabla (f^j) (r,\cdot,Y_r,Z_r),h\rangle_{H}+f_y^j(r,\cdot,Y_r,Z_r)\cdot \hat{Y}_r^{h} \\
&\label{eq:DY,h}+ \sum_{k=1}^d (f_z^j)^{:,k}(r,\cdot+\varepsilon h,Y_r, Z_r)\cdot (\hat{Z}_r^{h})^{:,k} \Big) dr -\int_s^T ((\hat{Z}_r^{h})^{j,:})^\top \cdot dW_r.
\end{align}
Using a priori estimates (see Proposition 3.2 in \cite{briand}) in $L^{p}$, we have for some constant $C_p$, independent of $\varepsilon$
\begin{align}
\label{eq:apriori1}
&\E\left[\sup_{s\in [t,T]} |(Y_s^{\varepsilon})^j-(\hat{Y}^{h}_s)^j|^p\right] + \E\left[\left(\int_t^T \|((Z_s^{\varepsilon})^{j,:})^\top-((\hat{Z}^{h}_s)^{j,:})^\top\|^2 ds\right)^{p/2}\right]\nonumber\\
&\leq C_p \left(\E\left[|(\xi^{\varepsilon})^{j}-\langle \nabla(\xi^{ j}),h\rangle_{H}|^p\right] +\E\left[\left(\int_0^T \abs{\tilde{A}_s^\varepsilon-\langle \nabla (f^{ j})(s,\cdot,Y_s,Z_s),h\rangle_{H}} ds \right)^p\right]\right) \nonumber\\
&\hspace{0.8em}+C_p\E\left[\left(\int_0^T \|\tilde{A}_s^{y,\varepsilon}-f_y^{ j}(s,\cdot,Y_s,Z_s)\| \|\hat{Y}_s^{h}\|ds\right)^p\right]\nonumber\\ 
&\hspace{0.8em}+ C_p\sum_{k=1}^d \E\left[\left(\int_0^T \| (f_z^j)^{:,k}(s,\cdot+\varepsilon h,Y_s \tr{\varepsilon h}, \widetilde{Z}^k_s)-(f_z^j)^{:,k}(s,\cdot+\varepsilon h,Y_s,Z_s)\|\|(\hat{Z}_s^{h})^{:,k} \|ds\right)^p \right].
\end{align}
Since $\xi^{j}$ is in $\D^{1,2}$, $\lim_{\varepsilon \to 0} \E\left[|(\xi^{\varepsilon})^j-\langle \nabla(\xi^{ j}),h\rangle_{H}|^p\right]=0$ by Lemma \ref{lemma:D12nec}. By Assumption ($H_1$), the second term in the right-hand side of \reff{eq:apriori1} goes to $0$ as $\varepsilon$ goes to $0$. For the last two terms, we will use Assumption ($H_2$). First, the above estimate implies directly that $(Y^{j}\tr{\varepsilon h}-Y^{j},(Z^{ j,:})^\top \tr{\varepsilon h}-(Z^{j,:})^\top)_\varepsilon$ goes to $0$ in $\S^q\times\H_{d}^q$ as $\varepsilon$ goes to $0$ for any $q\in(1,2)$. We can therefore conclude with Assumption ($H_2$), together with the fact that $\|f_y^{j}\|$ is bounded, that by the dominated convergence theorem:
\begin{align*}
&\E\left[\left(\int_0^T \|\tilde{A}_s^{y,\varepsilon}-f_y^{ j}(s,\cdot,Y_s,Z_s)\| \|\hat{Y}_s^{h}\| ds\right)^p\right]\\ 
&\leq C \E\left[\left(\int_0^T \|\tilde{A}_s^{y,\varepsilon}-f_y^{ j}(s,\cdot,Y_s,Z_s)\|^2 ds\right)^{p}\right]^\frac12 \E\left[\left(\int_0^T \|\hat{Y}_s^{h}\|^2ds\right)^{p}\right]^\frac12\underset{\varepsilon \to 0}{\longrightarrow}0.
\end{align*}
We can show similarly that the last term on the right-hand side of \reff{eq:apriori1} also goes to $0$, by using the fact that for any $j\in \{ 1,\ldots,n\}$, $((\hat Z^{h})^{ j,:})^\top\in\mathbb H_{d}^2$. It just remains to prove that for any $j\in \{1,\ldots,n \}$, $(\hat{Y}_t^{h})^{j}$ is a random operator on $H$ or equivalently that there exists $\mathcal{D}Y^j_t$ an $H$-valued r.v. such that $(\hat{Y}_t^{h})^{j}=\langle \mathcal{D} Y^j_t, h\rangle_H$ for any $h$ in $H$. To this end, let $(h_k)_k$ be an orthonormal system in $H$, we set for any $j \in \{1,\ldots, n \}, \ \ell\in \{1,\ldots,d \}$:
$$\mathcal{D} Y^j_t:=\sum_{k\geq 1} (\hat{Y}_t^{h_k})^j h_k,\ \mathcal{D} Z^{j,l}_s:=\sum_{k\geq 1} ((\hat{Z}_s^{h_k})^{j,l}) h_k.$$ 
Note that these elements are well-defined, since one can prove that $\mathcal{D} Y^j_t \in L^2(H)$ and that $\mathcal{D} Z^{j,\ell} \in L^2([t,T];H)$. Indeed using once again a priori estimates for affine BSDEs, there exists $C>0$ (which may differ from line to line) such that: 
\begin{align}
\label{eq:finitness}
&\E\left[\sum_{j=1}^n\|\mathcal{D} Y^j_t\|_H^2 + \sum_{j=1}^n\sum_{\ell=1}^d\int_t^T \|\mathcal{D} Z^{j,\ell}_s\|_H^2 ds\right]\nonumber\\
&=\sum_{k\geq 1} \E\left[\sum_{j=1}^n|(\hat{Y}_t^{h_k})^{ j}|^2 + \sum_{j=1}^n\sum_{\ell=1}^d\int_t^T |(\hat{Z}_s^{h_k})^{ j,\ell}|^2 ds\right]\nonumber\\
&\leq C \sum_{k\geq 1} \E\left[\sum_{j=1}^n|\langle \nabla (\xi^j), h_k\rangle_{H}|^2 + \sum_{j=1}^n\int_t^T |\langle \nabla (f^j)(s,Y_s,Z_s), h_k \rangle_{H}|^2 ds\right]\nonumber\\
&\leq C \sum_{j=1}^n\E\left[\|\nabla (\xi^j)\|_H^2 + \int_t^T \|\nabla (f^j)(s,Y_s,Z_s)\|_H^2 ds\right]<+\infty,
\end{align}
by our assumptions on $\xi$ and $f$. We now identify $(\hat{Y}^{h}_t)^{j}$ (respectively $(\hat{Z}_t^{h})^{j,\ell}$) with the inner product $\langle \mathcal{D} Y^{j}_t, h\rangle_H$ (respectively $\langle \mathcal{D} Z^{j,\ell}_t, h\rangle_H$).
For any $s\geq t$, it holds that:
\begin{align*}
\langle \mathcal{D} Y^{j}_s, h\rangle_H&=\sum_{k\geq 1} \langle \nabla (\xi^{ j}), h_k\rangle_{H} \langle h_k, h \rangle_H + \sum_{k\geq 1} \langle h_k, h \rangle_H \int_s^T \langle \nabla (f^{ j})(r,Y_r,Z_r), h_k\rangle_{H} dr\\
&\quad+ \sum_{k\geq 1} \langle h_k, h \rangle_H \int_s^T \left(f_y^{j}(r,Y_r,Z_r) \cdot \hat{Y}_r^{h_k} +\sum_{l=1}^d (f_z^j)^{:,l}(r,\cdot+\varepsilon h,Y_r, Z_r)\cdot (\hat{Z}_r^{h_k})^{:,l} \right)dr\\
&\quad+ \sum_{k\geq 1} \langle h_k, h \rangle_H \int_s^T ((\hat{Z}_r^{h_k})^{j,:})^\top\cdot  dW_r\\
&=  \langle  \nabla (\xi^{ j}), h\rangle_{H} +\int_s^T\Bigg( \langle \nabla (f^{ j})(r,Y_r,Z_r), h\rangle_H + f_y^{ j}(r,Y_r,Z_r) \cdot \sum_{k\geq 1}\hat{Y}_r^{h_k}\langle h_k,h\rangle_H \\
&\quad +\sum_{k\geq 1}  \langle h_k, h \rangle_H\sum_{l=1}^d  (f_z^j)^{:,l}(r,\cdot+\varepsilon h,Y_r, Z_r)\cdot (\hat{Z}_r^{h_k})^{:,l}\Bigg)dr \\
&\quad+ \int_s^T \Sum_{k\geq 1}\langle h_k,h\rangle_H((\hat Z^{h_k}_r)^{j,:})^\top \cdot dW_r,
\end{align*}
where we justify the exchange between the series and the Riemann integrals by Fubini's Theorem. Concerning the Wiener integral we make use of the stochastic Fubini's Theorem (see e.g. \cite{Veraar}) since by a priori estimates:
\begin{align*}
&\sum_{k\geq 1} \E\left[\int_0^T |\langle h_k,h\rangle_H|\times \|(\hat{Z}^{h_k})_t^{ j,:}\|^2 dt\right]^{1/2}\\
&\leq C \sum_{k\geq 1} \abs{\langle h_k,h\rangle_H}  \E\left[|\langle \nabla (\xi^j), h_k\rangle_H|^2 +\int_0^T |\langle \nabla (f^{j})(r,Y_r,Z_r), h_k\rangle_H|^2 dr\right]^{1/2}\\
&\leq C \E\left[\|\nabla (\xi^j)\|_H^2 +\int_0^T \|\nabla (f^{j})(r,Y_r,Z_r)\|_H^2 dr\right]<+\infty,
\end{align*}
where $C$ is a constant which may vary from line to line. By setting $$\mathcal{D} Y^h_s:=(\langle \mathcal{D} Y^{j}_s, h\rangle_H)_{j\in \{1,\ldots,n\}},\ \mathcal DZ_s^{h}:=\left(\langle \mathcal D Z^{j,l}_s,h\rangle_H\right)_{1\leq j\leq n,\ 1\leq \ell\leq d}, $$
we obtain
\begin{align*}
\langle \mathcal{D} Y^{j}_s, h\rangle_H&=  \langle  \nabla (\xi^{ j}), h\rangle_{H} +\int_s^T\Bigg( \langle \nabla (f^{ j})(r,Y_r,Z_r), h\rangle_H + f_y^{ j}(r,Y_r,Z_r)\cdot \mathcal{D} Y^h_s \\
&\quad +\sum_{\ell=1}^d  (f_z^j)^{:,\ell}(r,\cdot+\varepsilon h,Y_r, Z_r)\cdot (\mathcal D {Z}_r^h)^{:,\ell}  \Bigg)dr + \int_s^T  ((\mathcal{D} Z_r^h)^{j,:})^\top \cdot dW_r,
\end{align*}

Thus, by uniqueness of the solution to affine BSDEs with square integrable data, it holds that $(\hat{Y}_t^{h})^j=\langle \mathcal{D} Y^{ j}_t, h\rangle_H$ in $L^2(\real)$ and $(\hat{Z}^{h})^{j,:} \textbf{1}_{[t,T]}= (\mathcal{D} Z^{h})^{ j,:} $ in $\H_d^2$ for any $h$ in $H$.
Thus, using Estimate \eqref{eq:apriori1} we have proved that for any $h$ in $H$,
$$\lim_{\varepsilon \to 0} \E\left[|(Y_t^{\varepsilon})^j-\langle \mathcal{D} Y^{ j}_t,h\rangle_H|^p\right]=0.$$
Hence by Theorem \ref{th:newcarD}, $Y_t^{j}$ belongs to $\D^{1,p}$ and $\nabla Y_t^{j}=\mathcal{D} Y_t^{j}$. 

\vspace{0.5em}

If we set
$$ \mathcal{D} \int_t^T (Z_s^{j,:})^\top \cdot dW_s:= \sum_{k\geq 1} \int_t^T ((\hat{Z}_s^{h_k})^{ j,:})^\top \cdot  dW_sh_k,$$
the stochastic Fubini Theorem implies that:
$$ \left\langle \mathcal{D} \int_t^T (Z_s^{j,:})^\top \cdot dW_s, h\right\rangle_H =\int_t^T ((\mathcal{D} Z^{h}_s)^{j,:})^\top\cdot dW_s.  $$ 
Moreover, Burkholder-Davis-Gundy's inequality implies that there exists $\tilde{C}_p>0$ such that
\begin{align*}
&\E\left[\left|\varepsilon^{-1} \left(\int_t^T (Z_r^{ j,:})^\top \cdot dW_r \tr{\varepsilon h} - \int_t^T (Z_r^{ j,:})^\top\cdot dW_r\right)-\left\langle \mathcal{D} \int_t^T (Z_r^{j,:})^\top \cdot dW_r, h\right\rangle_H \right|^p\right]\\
&=\E\left[\left|\varepsilon^{-1} \left(\int_t^T (Z_r^{j,:})^\top \cdot dW_r \tr{\varepsilon h} - \int_t^T (Z_r^{j,:})^\top \cdot dW_r\right)-\int_t^T ((\mathcal{D} Z^{h}_r)^{j,:})^\top\cdot dW_r \right|^p\right]\\
&\leq \E\left[\sup_{t\leq s\leq T}\left|\int_t^s \left(\varepsilon^{-1}((Z_r^{j,:})^\top\tr{\varepsilon h} -(Z_r^{j,:})^\top)- ((\mathcal{D} Z^{h}_r)^{j,:})^\top\right)\cdot dW_r \right|^p\right]\\
&\leq \tilde{C}_{p} \E\left[\left(\int_t^T \|((Z_r^\varepsilon)^{j,:})^\top-  ((\mathcal{D} Z^{h}_r)^{j,:})^\top\|^2 dr \right)^{p/2}\right]\\
&=\tilde{C}_{p} \E\left[\left(\int_t^T \|((Z_r^\varepsilon)^{j,:})^\top-  ((\hat{Z}^{h}_r)^{j,:})^\top\|^2 dr \right)^{p/2}\right],
\end{align*}
where we have used in the last inequality the fact that $(\hat{Z}^{h})^{j,:} \textbf{1}_{[t,T]}= (\mathcal{D} Z^{h})^{ j,:} $ in $\H_d^2$ for any $h$ in $H$.

\vspace{0.5em}
The right-hand side above tends to $0$ as $\varepsilon$ goes to $0$, once again by \eqref{eq:apriori1}. Therefore, 
$$\int_t^T (Z_s^{ j,:})^\top\cdot dW_s\in\D^{1,p}\ \text{and}\ \nabla \int_t^T (Z_s^{ j,:})^\top\cdot  dW_s = \mathcal{D} \int_t^T (Z_s^{ j,:})^\top\cdot dW_s.$$ Furthermore, by the computations \eqref{eq:finitness} we deduce that $Y_t^j$ belongs to $\D^{1,2}$ and that $\int_t^T (Z_s^{j,:})^\top \cdot dW_s$ belong to $\D^{1,2}$ which, by \cite[Lemma 2.3]{pardouxpeng} implies that $Z^{ j,:}$ belongs to $L^2([t,T]; (\D^{1,2})^d)$.
Finally, to match with the notations of the papers \cite{pardouxpeng,EPQ} let us denote by $DY^{ j}$ and $DZ^{ j,:}$ the derivatives of respectively $\nabla Y^{j}$ and $\nabla Z^{ j,:}$. We also define $D_sY\in \mathbb R^{n\times d}$ by 
$$ (D_sY)^{j,\ell}:= (D_sY^j)^\ell, \; 1\leq j\leq n, \ 1\leq \ell\leq d,$$
and similarly for $D_s\xi$, $D_sf$ and $D_sZ^{:,k}$, $1\leq k\leq d$. Let us also define for any $1\leq j\leq n$, $D_sZ^{j,:}\in\R^{d\times n}$ by
$$(D_sZ^{j,:})^{k,\ell}:=(D_sZ^{j,k})^{\ell},\ 1\leq k\leq d,\ 1\leq \ell\leq n.$$
We thus obtain using the chain rule formula
\begin{align}
\label{eq:DY}
(D_s Y_t)^{j,\ell} =&\ (D_s \xi)^{j,\ell} +\int_t^T \left((D_s f)^{j,\ell}(r,Y_r,Z_r)+ f_y^{j}(r,Y_r,Z_r) \cdot ((D_s Y_r)^{j,:})^\top\right)dr \nonumber\\
& +\int_t^T \sum_{k=1}^d (f_z^j)^{:,k}(r ,Y_r, Z_r) \cdot ((D_s Z_r^{k,:})^{j,:})^\top dr -\int_t^T (D_s Z_r^{j,:})^{:,\ell}\cdot  dW_r,
\end{align}
which can be interpreted as an affine BSDE.
\qed
\end{proof}

\begin{Remark}
We would like to point out that since each process $Z^j$ is defined as a $\h$-valued r.v., one may be careful not to study $Z$ directly at a given time, as $Z_t$ is not well defined for a given $t$. Hence, in the proof we rather study at any time $t$ the random variable $\int_t^T Z_s dW_s$ and prove that it belongs to $\D^{1,2}$. Then by \cite[Lemma 2.3]{pardouxpeng} the latter result is equivalent to the fact that $Z$ belongs to $L^2([t,T]; \D^{1,2})$. 
\end{Remark}

\begin{Remark}
We emphasize that our criterion can also be used to study higher-order differentiability properties for $(Y,Z)$. For instance, the pair $(DY,DZ)$ is itself the solution of a (linear) BSDE. Therefore, as long as one is able to derive appropriate a priori estimates for this BSDE, the methodology above can then be applied to obtain conditions ensuring second-order Malliavin differentiability of $(Y,Z)$. Notice nonetheless that when handling higher order derivatives, products of lower order derivatives appear. One may then need to add conditions on the coefficients ensuring strong integrability properties of $(Y,Z)$ and their Malliavin derivatives. \end{Remark}

\section{Applications and discussion of the results}
\label{section:applidiscus}
For simplicity, in this section we will enforce that $n=d=1$.
\subsection{Application to FBSDEs} 
\label{section:FBSDE}

We consider in this section a FBSDE of the form

\begin{equation}\label{fbsde}
\begin{cases}
\displaystyle X_t= X_0+\int_0^t b(s,X_s) ds +\int_0^t \sigma(s,X_s)dW_s,\ t\in[0,T],\ \P_0-a.s.\\
\displaystyle Y_t = g(X_T) +\int_t^T f (s,X_s,Y_s,Z_s) ds -\int_t^T Z_s dW_s, \ t\in [0,T],\ \P_0-a.s.,
\end{cases} 
\end{equation}  
where $X_0\in \mathbb R$. We make the following Assumptions:

\begin{itemize}
\item[$(A_1)$]$b,\sigma : [0,T]\times \real \longrightarrow \real$ are continuous in time and continuously differentiable in space for any fixed time $t$ and such that there exist $k_b,k_\sigma >0$ with 
$$|b_x(t,x)|\leq k_b,\ |\sigma_x(t,x)|\leq k_\sigma, \text{ for all $x\in\R$}.$$ 
Besides $b(t,0), \sigma(t,0)$ are bounded functions of $t$.
\item[$(A_2)$]
\begin{itemize}
\item[(i)]$g$ is continuously differentiable with polynomial growth.
\item[(ii)]  $f : [0,T]\times \mathbb{R}^3 \longrightarrow \mathbb{R}$ is continuously differentiable in $(x,y,z)$ with bounded first partial derivatives in $y,z$ uniformly in $t$, such that $\E[\int_0^T \abs{f(s,0,0,0)}^2 ds] <+\infty$ and satisfying for some $C>0$
$$ \exists (q,\kappa)\in\R_+\times[0,2),\ |f_x(t,x,y,z)| \leq C(1+\abs{y}^\kappa+\abs{z}^\kappa+|x|^q), \ \forall (t,x,y,z) \in [0,T]\times \real^3.$$
\end{itemize}
\end{itemize}
Notice that under $(A_1)$ and $(A_2)$, the FBSDE \eqref{fbsde} admits a unique solution $(X,Y,Z)$ (see \cite{Pardoux_Peng90}). The well-known following lemma provides the existence of a Malliavin derivative for $X_t$ for all $t\in [0,T]$ under Assumption $(A_1)$ (see e.g. \cite[Theorem 2.2.1]{Nualartbook}).

\begin{Lemma}\label{lem:dx}
Under Assumption $(A_1)$, for any $p\geq 1$, $X_t \in \D^{1,p}$ for all $t\in [0,T]$, and $X\in \S^p$.\end{Lemma}
The following Proposition shows that Assumptions $(A_1)$ and $(A_2)$, which are actually weaker than $(PP1)$ and $(PP2)$, imply our new assumptions $(H_1)$ and $(H_2)$. As a corollary, using Theorem \ref{thm:yzd12}, we recover the original result \cite[Proposition 2.2]{pardouxpeng}.

\begin{Proposition}\label{prop:a:implique:h}
Let $(X,Y,Z)$ be the unique solution of FBSDE \eqref{fbsde}. Under Assumptions $(A_1)$, $(A_2)$, Assumptions $(L)$, $(D)$, $(H_1)$ and $(H_2)$ hold.
\end{Proposition}
\begin{Remark}
We insist on the fact that in the Markovian case, the original assumptions $($PP1$)$ and $($PP2$)$ of \cite{pardouxpeng} imply directly our new assumptions $(H_1)$ and $(H_2)$ while Assumptions $($EPQ1$)$ and $($EPQ2$)$ of \cite{EPQ} are strictly stronger than $(PP1)$ and $(PP2)$. In other words, in the Markovian case our assumptions are enough to recover the original result \cite[Proposition 2.2]{pardouxpeng}, without any additional conditions.
\end{Remark}
\begin{proof}[Proof of Proposition \ref{prop:a:implique:h}]
From Lemma \ref{lem:dx}, Property (D) holds by the chain rule formula and (L) follows from our assumptions. It remains to prove ($H_1$) and ($H_2$). We start with ($H_1$). Let $1<p<2$ and $h$ in $H$. Below $C$ denotes a positive constant which can differ from line to line. Recall that from our assumptions, 
\begin{equation}\label{eq:truc} \E\left[\sup_{t\in [0,T]} |Y_t|^r +\left(\int_0^T |Z_t|^2 dt\right)^{r/2} \right]<\infty, \quad \forall r\leq 2. 
\end{equation}

Denoting by $\bar{X}_t$ a random point between $X_t$ and $X_t\tr{\varepsilon h}$, where we suppressed the dependence on $\varepsilon$ for notational simplicity. We have for any $t$ in $[0,T]$, that 
\begin{align*}
&\E\left[\left|\varepsilon^{-1} (f(t,X_t\tr{\varepsilon h},Y_t,Z_t)-f(t,X_t,Y_t,Z_t)) - f_x(t,X_t,Y_t,Z_t) \langle DX_t, \dot{h}\rangle_{\h}\right|^p\right]\\
&=\E\left[\left|\frac{X_t\tr{\varepsilon h}-X_t}{\varepsilon} f_x(t,\bar{X}_t,Y_t,Z_t)- f_x(t,X_t,Y_t,Z_t) \langle DX_t, \dot{h}\rangle_{\h}\right|^p\right]\\
&\leq C \E\left[\left|\varepsilon^{-1} (X_t\tr{\varepsilon h}-X_t)-\langle DX_t, \dot{h}\rangle_{\h} \right|^p (1+|Y_t|^{\kappa p}+|Z_t|^{\kappa p}+|X_t|^{pq}+|X_t\tr{\varepsilon h}|^{pq})\right] \\
&\quad +C \E\left[\left| f_x(t,\bar{X}_t,Y_t,Z_t) -f_x(t,X_t,Y_t,Z_t)\right|^p \left|\langle DX_t, \dot{h}\rangle_{\h}\right|^p\right]\\
&\leq C \E\left[\left|\varepsilon^{-1} (X_t\tr{\varepsilon h}-X_t)-\langle DX_t, \dot{h}\rangle_{\h} \right|^{pr}\right]^{\frac1r}\E\left[(1+|Y_t|^{\kappa p}+|Z_t|^{\kappa p}+|X_t|^{pq}+|X_t\tr{\varepsilon h}|^{pq})^{\bar{r}}\right]^{\frac{1}{\bar{r}}} \\
&\quad +C \E\left[\left| f_x(t,\bar{X}_t,Y_t,Z_t) -f_x(t,X_t,Y_t,Z_t)\right|^p \left|\langle DX_t, \dot{h}\rangle_{\h}\right|^p\right]\\
&=:A_t^{1,\varepsilon} + A_t^{2,\varepsilon}, 
\end{align*}
where $\bar r>1$ and $p$ are chosen so that $p\kappa \bar r<2$ and $r$ denotes the H\"older conjugate of $\bar r$. Using the above estimates, we deduce
\begin{align*}
&\E\left[\left(\int_0^T\left|\varepsilon^{-1} (f(t,X_t\tr{\varepsilon h},Y_t,Z_t)-f(t,X_t,Y_t,Z_t)) - f_x(t,X_t,Y_t,Z_t) \langle DX_t, \dot{h}\rangle_{\h}\right|dt\right)^p\right]\\
&\leq \int_0^T\left(A_t^{1,\varepsilon} + A_t^{2,\varepsilon}\right)dt.
\end{align*}
Then, we have
\begin{align}
\label{eq:A1temp}
\int_0^TA_t^{1,\varepsilon}dt\leq &\ C\left(\int_0^T\E\left[\left|\varepsilon^{-1} (X_t\tr{\varepsilon h}-X_t)-\langle DX_t, \dot{h}\rangle_{\h} \right|^{pr}\right]^{2/r}dt\right)^{1/2}\nonumber\\
&\times\ \left(\int_0^T\E\left[(1+|Y_t|^{\kappa p}+|Z_t|^{\kappa p}+|X_t|^{pq}+|X_t\tr{\varepsilon h}|^{pq})^{\bar{r}}\right]^{2/\bar{r}}dt\right)^{1/2}.
\end{align}
In addition by Lemma \ref{lemma:shift2}, we have that $M^{\varepsilon,h}:=X \tr{\varepsilon h}-X$ is solution to the linear SDE:
$$ dM^{\varepsilon,h}_t = M^{\varepsilon,h}_t (b_x(t,\underline{X}_t) dt+ \sigma_x(t,\underline{X}_t) dW_t) + \varepsilon \sigma(t,X_t \tr{\varepsilon h}) \dot{h}_t dt,$$
where $\underline{X}_s$ denotes once again a random point between $X_s$ and $X_s\tr{\varepsilon h}$.  
Hence using Assumption $(A_1)$ and standard estimates for SDEs, we get that for any $q\geq1$, $$\lim_{\varepsilon \to 0} \E\left[\sup_{t\in [0,T]} |X_t\tr{\varepsilon h}-X_t|^q\right]=0.$$
Following the same lines as above, and recalling that $N^{h}:=\langle D X_t, \dot{h} \rangle_\h$ is solution to the SDE:
$$ dN^{h}_t = N^{h}_t (b_x(t,X_t) dt+ \sigma_x(t,X_t) dW_t) + \sigma(t,X_t) \dot{h}_t dt,$$
we get that the process $P^{\varepsilon,h}:=\varepsilon^{-1} (X \tr{\varepsilon h}-X) -\langle D X, \dot{h} \rangle_\h$ is solution to the affine SDE:
$$dP^{\varepsilon,h}_t=dH^\varepsilon_t+ P^{\varepsilon,h}_t (b_x(t,\underline{X}_t) dt +\sigma_x(t,\underline{X}_t) dW_t),$$
with 
\begin{align*}
dH^\varepsilon_t:=&\left(\langle D X_t, \dot{h} \rangle_{\h} (b_x(t,\underline{X}_t)-b_x(t,X_t)) + \dot{h}_t (\sigma(t,X_t\tr{\varepsilon h})-\sigma(t,X_t))\right)dt\\
&+\langle D X_t, \dot{h} \rangle_{\h} (\sigma_x(t,\underline{X}_t)-\sigma_x(t,X_t)) dW_t
\end{align*}
Using the fact that $\sigma_x, b_x$ are bounded, $\sigma$ has linear growth and is continuous, we get by similar computations than those done several times in this paper that:
$$ \lim_{\varepsilon \to 0}\E\left[\sup_{t\in [0,T]} |H^\varepsilon_t|^q\right]=0, \quad \forall q\geq 1, $$
from which we deduce using the explicit representation of solutions to affine SDEs (see e.g. \cite[Theorem V.9.53]{Protter}) that 
$$ \lim_{\varepsilon \to 0}\E\left[\sup_{t\in [0,T]} |\varepsilon^{-1} (X \tr{\varepsilon h}-X) -\langle D X_t, \dot{h} \rangle_\h|^q\right]=0, \quad \forall q\geq 1. $$
As a consequence, combining this estimate with \eqref{eq:A1temp}, we get that:
$$ \int_0^TA_t^{1,\varepsilon}dt\leq C\left(\E\left[\sup_{t\in [0,T]} \left|\varepsilon^{-1} (X_t\tr{\varepsilon h}-X_t)-\langle DX_t, \dot{h}\rangle_{\h} \right|^{pr}\right]^{2/r}\right)^{1/2},$$
which goes to $0$ as $\varepsilon$ goes to $0$, since we recall that we have chosen $p,\bar r>1$ so that $\kappa p \bar{r}<2$, which implies by \reff{eq:truc}, Lemma \ref{lem:dx} and the Cameron-Martin formula that
$$ \int_0^T\E\left[(1+|Y_t|^{\kappa p}+|Z_t|^{\kappa p}+|X_t|^{pq}+|X_t\tr{\varepsilon h}|^{pq})^{\bar{r}}\right]^{2/\bar{r}}dt<\infty. $$
Concerning the term $\int_0^T A_t^{2,\varepsilon}dt$, choosing $\tilde{p}>1$ so that $p\tilde{p}<2$, it holds by H\"older and by Jensen inequalities that
\begin{align*}
\int_0^T A_t^{2,\varepsilon}dt \leq C \left(\int_0^T \E\left[\left| f_x(t,\bar{X}_t,Y_t,Z_t) -f_x(t,X_t,Y_t,Z_t)\right|^{p\tilde{p}}\right] dt\right)^{1/\tilde{p}},
\end{align*}
since 
$$\E\left[\sup_{t\in [0,T]} \left|\langle DX_t, \dot{h}\rangle_{\h}\right|^q\right]<\infty, \quad \forall q>1.$$
As $$\lim_{\varepsilon \to 0} \E\left[\sup_{t\in [0,T]} |X_t\tr{\varepsilon h}-X_t|^q\right]=0, \quad \forall q \geq 1$$ it holds that 
$$\lim_{\varepsilon \to 0} \left| f_x(t,\bar{X}_t,Y_t,Z_t) -f_x(t,X_t,Y_t,Z_t)\right|^{p\tilde{p}}=0, \; \P_0\otimes dt-a.e.$$
Furthermore, for any $2>\rho>1$,
\begin{align*}
&\sup_{\varepsilon \in (0,1)} \int_0^T \E\left[\left| f_x(t,\bar{X}_t,Y_t,Z_t) -f_x(t,X_t,Y_t,Z_t)\right|^{\rho p\tilde{p}}\right] dt\\
&\leq C \sup_{\varepsilon \in (0,1)} \int_0^T \E\left[\left( 1 +|X_t|^q +|X_t \tr{\varepsilon h}|^q + |Y_t|^\kappa +|Z_t|^{\kappa} \right)^{\rho p\tilde{p}}\right] dt<\infty,
\end{align*} 
by choosing $p$ small enough so that $\kappa \rho p\tilde{p}\leq 2$. So by Lebesgue's dominated convergence theorem, 
$$ \lim_{\varepsilon \to 0} \int_0^T A_t^{2,\varepsilon}dt=0,$$
which proves ($H_1$). Concerning, ($H_2$) we just mention that $f_y$ (respectively $f_z$) is bounded, jointly continuous in $(x,y,z)$ and we make use of Lemma \ref{lemma:ctsH}.
\qed
\end{proof}

\subsection{Affine BSDEs}

The aim of this section is to prove that with our condition, we can provide weaker conditions compared to \cite{EPQ} for affine BSDEs. We take a driver of the form $$f(t,\omega,y,z):=\alpha_t(\omega) + \beta_t(\omega) y +\gamma_t(\omega) z$$ with bounded $\F$-progressively measurable processes such that $\alpha,\beta,\gamma\in L^2([0,T]; \mathbb D^{1,2})$, and $\xi$ in $\D^{1,2}$. The conditions given in \cite[Proposition 5.3]{EPQ} for proving that the associated solution $(Y,Z)$ is Malliavin differentiable read as follows (together with some measurability conditions): 
\begin{equation}
\label{eq:linsta}
\exists \eta>0 \textrm{ such that } \quad \E[|\xi|^{2+\eta}]<+\infty \quad \textrm{  \underline{and}  } \quad \int_0^T \E\left[\left(\int_\theta^T |K_\theta(s)|^2 ds\right)^{2+\eta}\right]^{1/(2+\eta)} d\theta<+\infty, 
\end{equation}
with $K_\theta(s):=|D_\theta \beta(s)|+|D_\theta \gamma(s)|$.

\vspace{0.5em}
In our setting, one needs to check assumptions (L), (D), ($H_1$) and ($H_2$). As mentioned below by Lemma \ref{lemma:ctsH} Condition ($H_2$) comes for free, and Assumptions (D) and (L) are also trivially satisfied. The interesting point is that ($H_1$) is true as soon as \eqref{eq:linsta} is replaced with: 
\begin{equation}
\label{eq:linstanous}
\exists \eta>0 \textrm{ such that } \lim_{\varepsilon \to 0} \int_0^T \E\left[\left|\varepsilon^{-1} (\mu_t\tr{\varepsilon h} -\mu_t)-\langle D\mu_t,\dot{h}\rangle_\h \right|^{2+\eta}\right] dt=0, \textrm{ for } \mu \in \{\beta,\gamma\}. 
\end{equation}   
Hence our condition only involves a condition on $\gamma$ and $\beta$ and not on $\xi$. For instance if $\beta$ and $\gamma$ are given as: 
$$ \beta_t=\varphi_1(X_t), \; \gamma_t:=\varphi_2(X_t), \quad t\in [0,T],$$
with $\varphi_1, \varphi_2$ two smooth functions with polynomial growth and $X$ is the solution to an SDE of the form of the one considered in Section \ref{section:FBSDE}, then the requirements of Conditions \eqref{eq:linsta} and \eqref{eq:linstanous} are satisfied for $\beta$ and $\gamma$, however in contradistinction to Condition \eqref{eq:linsta}, Assumption \eqref{eq:linstanous} does not put extra regularity on the terminal condition $\xi$.

\vspace{0.5em}
We make precise our result.
\begin{Proposition}
Let $\xi$ in $\D^{1,2}$, and $\alpha, \beta, \gamma$ bounded $\F$-progressively processes in $L^{2}([0,T];\D^{1,2})$ such that $D \alpha, D \beta$ and $D \gamma$ are $\F$-progressively measurable. Assume that Assumption \eqref{eq:linstanous} is in force. Then for any $t$ in $[0,T]$, $Y_t$ belongs to $\D^{1,2}$, $Z \in L^2([t,T];\D^{1,2})$ where $(Y,Z)$ is the unique solution in $\S^2\times \H^2$ to the affine BSDE:
$$ Y_t=\xi+\int_t^T (\alpha_s +\beta_s Y_s +\gamma_s Z_s) ds -\int_t^T Z_s dW_s, \quad t\in [0,T]. $$  
\end{Proposition}
 
\begin{proof}
Once again we check that assumptions of Theorem \ref{th:newcarD} are in force. Properties (D) and (L) are immediately satisfied. Let $f(t,\omega,y,z):=\alpha_t(\omega) + \beta_t(\omega) y +\gamma_t(\omega) z$. Since $f_y(t,\omega,y,z)=\beta_t(\omega)$, and $f_z(t,\omega,y,z)=\gamma_t(\omega)$ we immediately get by Lemma \ref{lemma:ctsH} and since $\beta,\gamma$ are bounded that ($H_2$) is satisfied. Concerning ($H_1$), we have for any $1<p<2$ and $h$ in $H$, that 
\begin{align}
\label{eq:templin}
&\E\left[\left(\int_0^T |\varepsilon^{-1} (f(t,\cdot+\varepsilon h,Y_t,Z_t)-f(t,\cdot,Y_t,Z_t)) - \langle Df(t,\cdot,Y_t,Z_t),\dot{h}\rangle_{\h}| dt\right)^p\right]\nonumber\\
&\leq C \E\left[\left(\int_0^T |\varepsilon^{-1} (\alpha_t\tr{\varepsilon h} -\alpha_t)-\langle D\alpha_t,\dot{h}\rangle_{\h}| dt\right)^p\right]\nonumber\\
&\quad+ C\E\left[\left(\int_0^T |Y_t (\varepsilon^{-1} (\beta_t\tr{\varepsilon h} -\beta_t)-\langle D\beta_t,\dot{h}\rangle_{\h})| dt\right)^p \right]\nonumber\\
&\quad+ C\E\left[\left( \int_0^T |Z_t (\varepsilon^{-1} (\gamma_t\tr{\varepsilon h} -\gamma_t)-\langle D\gamma_t,\dot{h}\rangle_{\h})| dt\right)^p\right]\nonumber\\
&=:A_1^\varepsilon+A_2^\varepsilon+A_3^\varepsilon,
\end{align}
where $C$ is a constant. By Lemma \ref{lemma:L12} we have that $\lim_{\varepsilon \to 0} A_1^\varepsilon=0$. 
We consider the term $A_3^\varepsilon$. We have that:
\begin{align*}
&A^\varepsilon_3 \leq \E\left[\int_0^T \left| Z_t (\varepsilon^{-1} (\gamma_t\tr{\varepsilon h} -\gamma_t)-\langle D\gamma_t,\dot{h}\rangle_{\h}) \right|^p dt\right]\\
&\leq C\int_0^T \E\left[|Z_t|^{2}\right]^{p/2} \E\left[\left|\varepsilon^{-1} (\gamma_t\tr{\varepsilon h} -\gamma_t)-\langle D\gamma_t,\dot{h}\rangle_\h \right|^{\frac{2p}{2-p}}\right]^{\frac{2-p}{2}} dt\\
&\leq C\left(\int_0^T \E\left[|Z_t|^{2}\right] dt\right)^{p/2} \left(\int_0^T \E\left[\left|\varepsilon^{-1} (\gamma_t\tr{\varepsilon h} -\gamma_t)-\langle D\gamma_t,\dot{h}\rangle_\h \right|^{\frac{2p}{2-p}}\right] dt\right)^{\frac{2-p}{2}}.
\end{align*}
Choosing $p$ such that $\frac{2p}{2-p}=2+\eta$ we get that $A_3^\varepsilon$ converges to $0$ as $\varepsilon$ tends to $0$ by \eqref{eq:linstanous}. Similarly, $\lim_{\varepsilon \to 0} A_2^\varepsilon=0$ for this choice of $p$.
\qed
\end{proof} 

\begin{Remark}
Note that, since the BSDE is affine, $Y_t$ can be expressed explicitly as:
$$ Y_t=\E\left[\left. M_{t,T} \xi -\int_t^TM_{t,s} \alpha_s ds \right|\mathcal{F}_t\right],$$
where
$$M_{t,s}:=\exp\left(\int_t^s \gamma_u dW_u -\frac12 \int_t^s |\gamma_u|^2 du +\int_t^s \beta_u du\right),\ s\in[t,T].$$
Hence, on the one hand, $Y_t$ belongs to $\D^{1,2}$ if and only if the coefficients $\alpha, \beta, \gamma$ belong to $L^2([0,T];\D^{1,2})$ and $\xi$ is in $\D^{1,2}$. The same conclusion follows for the $Z$ component. Hence, neither our condition \eqref{eq:linstanous} nor the one of \cite{EPQ}, namely \eqref{eq:linsta}, are sharp. However, both are sharp in the case where $\beta=\gamma=0$. On the other hand, Conditions \eqref{eq:linsta} or \eqref{eq:linstanous} give more information than the simple fact that $Y, Z$ are Malliavin differentiable, since they imply that the BSDE solved by $(DY,DZ)$ is limit in $\S^2\times \H^2$ of respectively $(DY^n,DZ^n)$ $($where $(Y^n,Z^n)$ is the solution to the Picard iteration equation at order $n$ approximating $(Y,Z))$ for \eqref{eq:linsta}, and of the difference quotient $(\varepsilon^{-1}(Y\tr{\varepsilon h}-Y),\varepsilon^{-1}(Z\tr{\varepsilon h}-Z))$ in our case \eqref{eq:linstanous}.   
\end{Remark}

\subsection{Discussion and comparison of the results}
\label{section:discu}

We would like before going to the quadratic BSDE case to make a comment about the difference between our approach and the one of \cite{pardouxpeng,EPQ} and our approach. In these references, the authors consider the sequence of BSDEs:
$$ Y_t^n =\xi+\int_t^T f(s,Y_s^{n-1},Z_s^{n-1}) ds -\int_t^T Z_s^n dW_s, \quad t\in [0,T], $$
which approximate in $\S^2\times \H^2$ the solution to the original BSDE:
$$ Y_t =\xi+\int_t^T f(s,Y_s,Z_s) ds -\int_t^T Z_s dW_s, \quad t\in [0,T].$$
Now, under mild assumptions on $f$, the processes $(Y^n,Z^n)$ are Malliavin differentiable and it holds that a version of $(D_r Y_t^n,D_r Z_t^n)$ satisfies for $t\in [0,T], \; r\leq t$:
\begin{align*}
 D_r Y_t^n =&\ D_r\xi+\int_t^T[ D_rf(s,\Theta_s^{n-1}) + \partial_y f(s,\Theta_s^{n-1}) D_rY_s^{n-1}+\partial_z f(s,\Theta_s^{n-1}) D_r Z_s^{n-1}]ds \\
 &-\int_t^T D_rZ_s^n dW_s,
 \end{align*}
with $\Theta_s^{n-1}:=(Y_s^{n-1},Z_s^{n-1})$.
On the other if $(Y,Z)$ where Malliavin differentiable we would have that a version of $(D_r Y_t,D_r Z_t)$ would satisfy for $t\in [0,T], \; r\leq t$:
$$ D_r Y_t =D_r\xi+\int_t^T[ D_rf(s,Y_s,Z_s) + \partial_y f(s,Y_s,Z_s) D_rY_s+\partial_z f(s,Y_s,Z_s) D_r Z_s]ds -\int_t^T D_rZ_s dW_s. $$
In other words, assuming that $\partial_y f$ and $\partial_zf$ to be continuous, we would get formally that $(DY^n,DZ^n)$ converges to $(DY,DZ)$ (in $\S^2\times\H^2$) as $n$ goes to infinity provided that at the limit one can replace $D_rf(s,Y_s^{n-1},Z_s^{n-1})$ by $D_r f(s,Y_s,Z_s)$ which is exactly where comes the main assumption in \cite{EPQ,pardouxpeng} which impose $D_r f$ to be (stochastic) Lipschitz continuous in $(y,z)$ with integrability conditions on the Lipschitz constant to make the aforementioned argument rigorous. However, it is not a necessary condition for $(Y,Z)$ to be Malliavin differentiable that $(DY^n,DZ^n)$ to converge to $(DY,DZ)$. However, for $Y_t$ to be in $\D^{1,2}$, it is necessary (and sufficient) that $\varepsilon^{-1} (Y_t\tr{\varepsilon h}-Y_t)$ converges in $L^p$ for some $p<2$ to $\langle DY_t,\dot{h}\rangle_\h$ for any $h$ in $H$ (according to Theorem \ref{th:newcarD}). Hence, this is an advantage of our conditions.   

\section{Extension to quadratic growth BSDEs}
\label{section:quadratic}
The aim of this section is to extend our previous results to so-called quadratic growth BSDEs. Some results for these equations already exist in the literature, see in particular \cite{ank,irr} or the thesis \cite{dosreis}, however they are generally limited to specific forms of the generators or to a Markovian setting. We will show that our approach to the Malliavin differentiability is flexible enough to be able to treat this problem without major modifications to our proofs. Since the wellposedness theory for multidimensional quadratic BSDEs is still an open problem, we enforce $n=1$ throughout this section.

\vspace{0.5em}
We will now list our assumptions in this quadratic setting
 \begin{itemize}
\item[($D_\infty$)] $\xi$ is bounded, belongs to $\D^{1,\infty}$ and its Malliavin derivative $D\xi$ is bounded, for any $(y,z)\in\R\times \R^d$, $(t,\omega)\longmapsto f(t,\omega,y,z)$ is in $L^2([0,T];\D^{1,\infty})$, $f(\cdot,y,z)$ and $Df(\cdot,y,z)$ are $\F$-progressively measurable, $Df(\cdot,y,z)$ is uniformly bounded in $y,z$.
\item[(Q)] The map $(y,z)\longmapsto f(\cdot,y,z)$ is continuously differentiable and there exists some constant $C>0$  such that for any $(s,\omega,y,z,z')\in[0,T]\times\Omega\times \R\times \R^d\times \R^d$
$$\abs{f(s,\omega,y,z)-f(s,\omega,y,z')}\leq C\left(1+\|z\|+\|z'\|\right)\|z-z'\|,\ \abs{f(s,\omega,0,0)}\leq C,$$
$$\abs{f_y(s,\omega,y,z)}\leq C,\  \|f_z(s,\omega,y,z)\|\leq C(1+\|z\|),$$
where $f_z=\left( \frac{\partial f}{\partial z_{l}}\right)_{l \in \{1,\ldots, d\}}$ denotes the gradient of $f$ with respect to the $z$ variable.
\item[($H_{1,\infty}$)] For any $p>1$ and for any $h\in H$
$$\lim\limits_{\varepsilon \to 0}\ \E\left[ \left(\int_0^T\left|\frac{f(s,\cdot+\varepsilon h, Y_s, Z_s)-f(s,\cdot,Y_s,Z_s)}{\varepsilon}-\langle Df(s,\cdot,Y_s,Z_s),\dh\rangle_{\h}\right| ds\right)^p\right]=0.$$
\item[($H_{2,\infty}$)] Let $(\varepsilon_k)_{k\in \N}$ be a sequence in $(0,1]$ such that $\lim\limits_{k\to +\infty} \varepsilon_k=0$, and let  $(Y^k,Z^k)_k$ be a sequence of random variables which converges in $\S^p\times \H_d^p$ for any $p>1$ to some $(Y,Z)$. Then for all $h\in H$, the following convergences hold in probability
\begin{align}
\label{assumpcts1bis}
&\No{f_y(\cdot,\omega+\varepsilon_k h,Y^k_\cdot,Z_\cdot) -f_y(\cdot,\omega,Y_\cdot,Z_\cdot)}_{L^2([0,T];\R)}\underset{k\to+\infty}{\longrightarrow}0 \nonumber\\
& \No{f_z(\cdot,\omega+\varepsilon_k h,Y^k_\cdot,Z^k_\cdot)-  f_z(\cdot,\omega,Y_\cdot,Z_\cdot)}_{L^2([0,T];\R^d)}\underset{n\to+\infty}{\longrightarrow}0,\end{align}
or
\begin{align}
\label{assumpcts2bis}
\No{f_y(\cdot,\omega+\varepsilon_k h,Y^k_\cdot,Z^k_\cdot) -f_y(\cdot,\omega,Y_\cdot,Z_\cdot)}_{L^2([0,T];\R)}\underset{k\to +\infty}{\longrightarrow}0\nonumber\\
 \No{f_z(\cdot,\omega+\varepsilon_k h,Y_\cdot,Z^k_\cdot)   -f_z(\cdot,\omega,Y_\cdot,Z_\cdot)}_{L^2([0,T];\R^d)}\underset{k\to +\infty}{\longrightarrow}0.
\end{align}
\end{itemize}

Let $\S^\infty$ be the set of $\F$-progressively measurable processes $Y$ such that $\sup_{t\in [0,T]} |Y_t|$ is bounded and $\H^2_{\rm{BMO}}$ the set of $\R^d$-valued predictable processes $Z$ such that:
$$ \underset{\tau \in \mathcal{T}}{\rm essup}\ \E\left[\left.\int_\tau^T \|Z_s\|^2 ds \right| \mathcal{F}_\tau\right]<+\infty, \ \P_0-a.s.,$$
where $\mathcal{T}$ denotes the set of $\F$-stopping times with values in $[0,T]$. We start by recalling the following by now classical results on quadratic growth BSDEs and stochastic Lipschitz BSDEs, which can be found among others in \cite{irr}.
\begin{Proposition}\label{prop:BMO}
Under Assumptions $(D_\infty)$ and $(Q)$, the BSDEs \reff{bsde} and \reff{eq:DY,h} both admit a unique solution in $\S^\infty\times\H^2_{\rm{BMO}}$. 
\end{Proposition}

We have the following extension of Theorem \ref{thm:yzd12}.
\begin{Theorem}\label{thm:yzd12quad}
Let $t$ be in $[0,T]$. Under Assumptions $(D_\infty)$, $(Q)$, $(H_{1,\infty})$ and $(H_{2,\infty})$, $Y_t$ belongs to $ \D^{1,\infty}$ and $Z \in L^\infty([t,T];(\D^{1,2})^d)$.
\end{Theorem}

\begin{proof}
We follow the proof of Theorem \reff{thm:yzd12}, using the same notations. Since the BSDEs are now quadratic, we can use the {\it a priori} estimates of Lemma $A.1$ in \cite{irr} to obtain that for any $p>1$, there exists some $q>1$ such that
\begin{align}
\label{eq:apriori1quad}
&\E\left[\sup_{s\in [t,T]} |Y_s^\varepsilon-\hat{Y}^h_s|^{2p}\right] + \E\left[\left(\int_t^T \|Z_s^\varepsilon-\hat{Z}^h_s\|^2 ds\right)^{p}\right]\nonumber\\
&\leq C_p \left(\E\left[|\xi^\varepsilon-\langle D\xi,\dh\rangle_{\h}|^{pq}\right]^{1/q} +\E\left[\left(\int_0^T \tilde{A}_s^\varepsilon-\langle Df(s,\cdot,Y_s,Z_s),\dh\rangle_{\h} ds \right)^{pq}\right]^{1/q}\right) \nonumber\\
\nonumber&\hspace{0.8em}+C_p\E\left[\left(\int_0^T \abs{\tilde{A}_s^{y,\varepsilon}-f_y(s,\cdot,Y_s,Z_s)} \abs{\hat{Y}_s^h}ds\right)^{pq}\right]^{1/q}\\
&\hspace{0.8em}+C_p\E\left[\left(\int_0^T \|\tilde{A}_s^{z,\varepsilon}-f_z(s,\cdot,Y_s,Z_s)\|\|\hat{Z}_s^h\|ds\right)^{pq} \right]^{1/q},
\end{align}
where we set
\begin{align*}
    &\tilde{A}_s^{y,\varepsilon}:=f_y(s,\cdot+\varepsilon h, \bar{Y}_s^{\varepsilon,h}, Z_s)\\
  & \tilde{A}_s^\varepsilon:=\frac1\varepsilon (f(r,\cdot+\varepsilon h, Y_s, Z_s)-f(s,\cdot,Y_s,Z_s)),
 \end{align*}
where $\bar{Y}_r^{\varepsilon,h}$ is a convex combination of $Y_r$ and $Y_r\circ\tau_{\varepsilon h}$ and 
$$\tilde{A}_s^{z,\varepsilon}:=((\tilde{A}_s^{z,\varepsilon})^k)_{k\in \{1,\ldots,d\}}$$ with $(\tilde{A}_s^{z,\varepsilon})^k:=f_z(s,\cdot+\varepsilon h, Y_s \tr{\varepsilon h}, \widetilde Z^k_r)$, and $\widetilde Z^k_r$ is as in the proof of Theorem \reff{thm:yzd12}.

\vspace{0.5em}
Since $\xi\in\D^{1,\infty}$, the first term on the right-hand side above goes to $0$ thanks to Theorem \ref{th:newcarD}. Moreover, the second term also goes to $0$ thanks to Assumption $(H_\infty)$. Then, since $f_y$ is bounded by Assumption $(Q)$ and since $\tilde Y^h\in \S^\infty$ by Proposition \ref{prop:BMO}, we can easily conclude with Assumption $(H_2)$ and the dominated convergence theorem that the third term on the right-hand side also goes to $0$. Let us now concentrate on the fourth term involving the control variable. By Cauchy-Schwarz inequality we have that
\begin{align}
\label{eq:quadratemp1}
&\E\left[\left(\int_0^T \|\tilde{A}_s^{z,\varepsilon}-f_z(s,\cdot,Y_s,Z_s)\|\|\hat{Z}_s^h\|ds\right)^{pq} \right]\nonumber\\
&\leq \E\left[\left(\int_0^T \|\tilde{A}_s^{z,\varepsilon}-f_z(s,\cdot,Y_s,Z_s)\|^2 ds\right)^{pq}\right]^{1/2} \E\left[\left(\int_0^T \|\widetilde{Z}_s^h\|^2 ds\right)^{pq} \right]^{1/2}.
\end{align}
Since $(\hat{Y}^h,\hat{Z}^h)$ is the solution to the stochastic linear BSDE \eqref{eq:DY,h} with bounded coefficients $Df$ and $f_y$ (by $(D_\infty)$) and $f_z(s,Y_s,Z_s)$ is in $\H^2_{\rm{BMO}}$ since $\|f_z(s,Y_s,Z_s)\|\leq C(1+\|Z_s\|)$ (by Assumption $(Q)$), we deduce that $\hat{Z}^h\in \H^2_{\rm{BMO}}$ which implies that $\hat{Z}^h \in \H_d^m$ for any $m>1$ by the energy inequalities. Furthermore, for any $\eta>0$ it holds that
\begin{align*}
&\E\left[\left(\int_0^T \|\tilde{A}_s^{z,\varepsilon}-f_z(s,\cdot,Y_s,Z_s)\| \|\hat{Z}_s^h\| ds\right)^{pq+\eta} \right]\\
 & \leq C\E\left[\left(\int_0^T\left(1+\|Z_s\|+\|Z_s\tr{\varepsilon h}\|\right)\|\hat{Z}_s^h\| ds\right)^{pq+\eta}  \right]\\
 &\leq C\E\left[\left(\int_0^T\left(1+\|Z_s\|+\|Z_s\tr{\varepsilon h}\|\right)^2 ds \right)^{\frac{pq+\eta}{2}}\left( \int_0^T \|\hat{Z}_s^h\|^2 ds\right)^{\frac{pq+\eta}{2}} \right]\\
 &\leq  C\E\left[ \left(\int_0^T(1+\|Z_s\|+\|Z_s\tr{\varepsilon h}\|)^2 ds \right)^{pq+\eta}\right]^{1/2} \E\left[\left( \int_0^T \|\hat{Z}_s^h\|^2 ds\right)^{pq+\eta} \right]^{1/2}\\
 &\leq C\left( 1+ \E\left[ \left(\int_0^T \|Z_s\|^2 ds\right)^{p'}\right]^{1/q'}\right) \E\left[\left( \int_0^T \|\hat{Z}_s^h\|^2 ds\right)^{pq+\eta} \right]^{1/2}<+\infty,
 \end{align*} where $p', q'>1$ using H\"older Inequality and Proposition \ref{prop.cam}. Hence, taking limit as $\varepsilon$ goes to $0$ in \eqref{eq:quadratemp1} we get that $\lim_{\varepsilon \to 0} \E\left[\sup_{s\in [t,T]} |Y_s^\varepsilon-\hat{Y}^h_s|^{2p}\right] + \E\left[\left(\int_t^T \|Z_s^\varepsilon-\hat{Z}^h_s\|^2 ds\right)^{p}\right]=0$. Following the same lines as in the proof of Theorem \ref{thm:yzd12}, one can use a priori estimates for quadratic growth BSDEs to obtain that $\hat{Y}^h$ and $\hat{Z}^h$ are linear operators-valued r.v.. This proves that $Y_t$ and $\int_t^T Z_s^{ \top} \cdot dW_s$ belongs to $\D^{1,\infty}$ by Theorem \ref{th:newcarD}. In particular, $Z \textbf{1}_{[t,T]}$ belongs to $L^2([t,T];(\D^{1,2})^d)$ (see \cite{pardouxpeng}). Moreover, since $(D_tY,D_tZ)$ is the solution of the stochastic linear BSDE \eqref{eq:DY} for any $t\in [0,T]$ and Assumptions $(D_\infty)$ and $(Q)$ hold, from the relation $(D_t Y_t)^j=(Z_t)^j$ for any $j$ in $\{1,\ldots,d\}$, and for all $t\in[0,T]$ we obtain $Z \textbf{1}_{[t,T]}\in L^\infty([t,T];(\D^{1,2})^d)$.
 \qed
\end{proof}

\begin{Remark}
We would like to point out that our conditions cover the case of Markovian quadratic BSDEs presented in \cite[Theorem 2.9]{DosReis_Imkeller}. Indeed, assume that we consider a forward-backward system of the form \eqref{fbsde} where the solution process $X$ to the forward SDE is $m$-dimensional with $m$ a positive integer (so that we match with the notations and assumptions of \cite[Theorem 2.9]{DosReis_Imkeller}) under assumptions, $(D_\infty)$, $(Q)$, $(A_1)$, $(A_2) (i)$ and where $(A_2) (ii)$ is replaced by the following assumption:
\begin{itemize}
\item[$(A_2) (ii')$] $f : [0,T]\times \mathbb{R}^3 \longrightarrow \R^m\times \real \times \R^d$ is continuously differentiable in $(x,y,z)$ and satisfying for some $C>0$
$$ \exists q\in\R_+,\ \|f_x(t,x,y,z)\| \leq C(1+\abs{y}+\|z\|^2+\|x\|^q), \ \forall (t,x,y,z) \in [0,T]\times \R^m\times\R\times\R^d,$$
\end{itemize}
where $f_x:=\left(\frac{\partial f}{\partial x_{l}}\right)_{l \in \{1,\ldots, m\}}$ denotes the gradient of $f$ with respect to the variable $x$.
Under these assumptions, we can check that $(H_{1,\infty})$ and $(H_{2,\infty})$ are in force. To see this we just make a comment about how the proof of Proposition \ref{prop:a:implique:h} has to be modified to obtain $(H_{1,\infty})$, whereas $(H_{2,\infty})$ is met trivially. Using the notations of this proof one can manage a term of the form:
$$  \E\left[\left(\int_0^T \left\|\varepsilon^{-1}(X_t\tr{\varepsilon h}-X_t)-\langle D X_t, \dot{h} \rangle_\h\right\| \|f_x(t,\bar{X_t},Y_t,Z_t)\| dt\right)^p\right] $$
as follows:
\begin{align*}
&\E\left[\left(\int_0^T \left\|\varepsilon^{-1}(X_t\tr{\varepsilon h}-X_t)-\langle D X_t, \dot{h} \rangle_\h\right\| \|f_x(t,\bar{X_t},Y_t,Z_t)\| dt\right)^p\right]\\
&\leq C \E\left[\sup_{t\in [0,T]} \left\|\frac{X_t\tr{\varepsilon h}-X_t}{\varepsilon}-\langle D X_t, \dot{h} \rangle_\h\right\|^p \left(\int_0^T (1+\|X_t\|^q+\|X_t \tr{\varepsilon h}\|^q+|Y_t| + \|Z_t\|^2) dt\right)^p\right]\\
&\leq C \E\left[\sup_{t\in [0,T]} \left\|\varepsilon^{-1}(X_t\tr{\varepsilon h}-X_t)-\langle D X_t, \dot{h} \rangle_\h\right\|^{2p}\right]^{1/2} \\
&\quad \times\E\left[\left(\int_0^T (1+\|X_t\|^q+\|X_t \tr{\varepsilon h}\|^q+|Y_t| + \|Z_t\|^2) dt\right)^{2p}\right]^{1/2},
\end{align*}
which goes to $0$ as $\varepsilon$ goes to $0$ since $Z$ belongs to $\H^2_{\rm{BMO}}$ and since $Y$ is bounded. The term involving $A^{2,\varepsilon}$ can be treated similarly.
\end{Remark}

\appendix
\section{Appendix}

\numberwithin{equation}{section}

The following lemma was remarked in \cite[Remark 2]{Sugita} with the set of polynomial cylindrical functions $\mathcal P$, we provide a proof of it with the set of cylindrical functions $\mathcal S$.

\begin{Lemma}
\label{lemma:SugitabisStep1Step1}
Let $p>1$ and $F$ be in $L^p(\R)$, $G\in \mathcal S$ and $h\in H$. The mapping $\varepsilon \longmapsto \E[F \circ \tau_{\varepsilon h} G]$ is differentiable in $\varepsilon$ and 
\begin{equation}
\label{eq:step1}
\frac{d}{d\varepsilon}  \E[F \circ \tau_{\varepsilon h} \; G] = \E[F \circ \tau_{\varepsilon h} \, \delta(G h)].
\end{equation}
\end{Lemma}

\begin{proof}
Let $\eta>0$, by the Cameron-Martin formula, we have that 
\begin{align*}
&\eta^{-1} \left( \E[F \tr{(\eta+\varepsilon) h} G] - \E[F \tr{\varepsilon h} G] \right)\\
&= \E\left[F \tr{\varepsilon h} \frac{G\circ \tau_{-\eta h} \exp\left(\eta \int_0^T \dot{h}(u)^{\top} \cdot \ dW_u -\frac{|\eta|^2}{2} \int_0^T \|\dot{h}(u)\|^2 du\right)- G}{\eta}\right].
\end{align*}
Hence 
\begin{align*}
&\lim_{\eta\to 0} \eta^{-1} \left( \E[F \tr{(\eta+\varepsilon) h} G] - \E[F \tr{\varepsilon h} G] \right)\\ 
&= \E\left[F \tr{\varepsilon h} \lim_{\eta\to 0} \frac{G\circ \tau_{-\eta h} \exp\left(\eta \int_0^T \dot{h}(u)^{\top} \cdot \ dW_u -\frac{|\eta|^2}{2} \int_0^T \|\dot{h}(u)\|^2 du\right)- G}{s}\right]\\
&= \E\left[F \tr{\varepsilon h} \lim_{\eta\to 0} \left( \frac{G\circ \tau_{-\eta h} - G}{\eta} + G\circ \tau_{-\eta h} \frac{ \exp\left(\eta \int_0^T \dot{h}(u)^{\top} \cdot \ dW_u -\frac{|\eta|^2}{2} \int_0^T \|\dot{h}(u)\|^2 du\right)-1}{\eta}\right)\right],
\end{align*}
where the exchange between the limit and the expectation is justified by the fact that 
\begin{equation}\label{eq:unifint}
\sup_{\eta\in (0,1]} \eta^{-q} \E\left[\left| G\circ \tau_{-\eta h} \exp\left(\eta \int_0^T \dot{h}(u)^{\top} \cdot\ dW_u -\frac{|\eta|^2}{2} \int_0^T \|\dot{h}(u)\|^2 du\right)- G \right|^q\right]<+\infty
\end{equation} for any $q>1$ and by the Cameron-Martin formula. Indeed for any $r$ in $(1,p)$ we have by H\"older Inequality:
\begin{align*}
&\E\left[\left|F \tr{\varepsilon h}\frac{G\circ \tau_{-\eta h} \exp\left(\eta \int_0^T \dot{h}(u)^{\top} \cdot\ dW_u -\frac{|\eta|^2}{2} \int_0^T \|\dot{h}(u)\|^2 du\right)- G}{\eta}\right|^r\right]\\
&\leq \underbrace{\E\left[|F \tr{\varepsilon h}|^{p_1}\right]^{r/p_1}}_{=:E_1}\underbrace{\E\left[\abs{\frac{G\circ \tau_{-\eta h} \exp\left(\eta \int_0^T \dot{h}(u)^{\top} \cdot\ dW_u -\frac{|\eta|^2}{2} \int_0^T \|\dot{h}(u)\|^2 du\right)- G}{\eta}}^{r p_2}\right]^{1/p_2}}_{=:E_2},
\end{align*}
where $r<p_1<p$ and $p_2$ is the H\"older conjugate of $p_1/r$. Using Cameron-Martin Formula for $E_1$, Relation \eqref{exp.int}  and H\"older Inequality with $r_1=\frac{p}{p_1}$ and $r_2$ such that $\frac{1}{r_1}+\frac{1}{r_2}=1$, we deduce that:
$$ E_1\leq \E[\abs{F}^p]^{r/p} \E\left[ \abs{\mathcal{E}\left( \int_0^T \dot{h}_s^{\top} \cdot\ dW_s\right)}^{r_2}\right]^{1/r_2}<+\infty.$$

We now turn to $E_2$, for any $q>1$
\begin{align*}
&\sup_{\eta\in (0,1]}\E\left[\abs{\frac{G\circ \tau_{-\eta h} \mathcal{E}\left(\eta \int_0^T \dot{h}(u)^{\top} \cdot\ dW_u\right) - G}{\eta}}^{q}\right]\\
&\leq \underbrace{\sup_{\eta \in (0,1]}\E\left[\abs{\frac{G\circ \tau_{-\eta h}-G}{\eta} \mathcal{E}\left(\eta\int_0^T \dot{h}(u)^{\top} \cdot\ dW_u\right) }^{q}\right]}_{=:A_1}+ \underbrace{\sup_{\eta\in (0,1]}\E\left[\abs{\frac{ \mathcal{E}\left(\eta\int_0^T \dot{h}(u)^{\top}  \cdot \ dW_u\right) - 1}{\eta}G}^{q}\right]}_{=:A_2},
\end{align*}

hence, on the one hand there exists $\alpha_1,\alpha_2>1$ such that:
\begin{align*}
A_1&\leq \sup_{\eta \in (0,1]} \eta^{-q} \E\left[\abs{G\circ \tau_{-\eta h}-G}^{q\alpha_1}\right]^\frac{1}{\alpha_1} \E\left[\abs{\mathcal{E}\left(\eta \int_0^T \dot{h}(u)^{\top} \cdot\ dW_u\right) }^{q\alpha_2}\right]^\frac{1}{\alpha_2}<+\infty,
\end{align*} using the fact that $G$ is polynomial, so $G$ is locally Lipschitz and we conclude by Relation \eqref{exp.int}. On the other hand, using the mean value theorem and Relation \eqref{exp.int}, we obtain also $A_2<+\infty$. We deduce that Relation \eqref{eq:unifint} holds. Moreover, given that $G\in \mathcal{P}$ is polynomial, we deduce that $ \frac{G\circ \tau_{-\eta h} - G}{\eta}\underset{\eta\to 0}{\to}-\langle \nabla G,h\rangle_H$ a.s.. Hence, 
\begin{align*}
&\lim_{\eta\to 0} \eta^{-1} \left( \E[F \tr{(\eta+\varepsilon) h} G] - \E[F \tr{\varepsilon h} G] \right)\\ 
&= \E\left[F \tr{\varepsilon h} \lim_{\eta\to 0} \left( \frac{G\circ \tau_{-\eta h} - G}{\eta} + G\circ \tau_{-\eta h} \frac{ \exp\left(\eta \int_0^T \dot{h}(u)^{\top} \cdot \ dW_u -\frac{|\eta|^2}{2} \int_0^T \|\dot{h}(u)\|^2 du\right)-1}{\eta}\right)\right]\\
&= \E\left[F \tr{\varepsilon h} \left( - \langle \nabla G, h\rangle_{H} + G \delta(h)\right)\right]= \E\left[F \tr{\varepsilon h} \, \delta(G h)\right],
\end{align*}
by \eqref{eq:deltaprod}, so \eqref{eq:step1} holds.
\qed
\end{proof}

\begin{Lemma}
\label{lemma:L12}
Let $\alpha$ in $L^{2}([0,T];\D^{1,2})$. Then for any $p$ in $(1,2)$, 
$$\lim_{\varepsilon \to 0} \E\left[\int_0^T \left|\frac{\alpha_s \tr{\varepsilon h} - \alpha_s }{\varepsilon}-\langle \nabla \alpha_s,h\rangle_{H}\right|^p ds\right]=0.$$
\end{Lemma}

\begin{proof}
Note first that the space $L^{2}([0,T];\D^{1,2})$ can be identified with the space $\D^{1,2}(\h)$ which is the completion of the set of $\h$-valued r.v. of the form: 
$$ \sum_{i=1}^n F_i u_i, \quad F_i \in \mathcal{S}, \; u_i \in L^2([0,T]), \; n \geq 1, $$
with respect to the norm $\|\cdot\|_{1,2,2}$ defined as:
$$ \|u\|_{1,2,2}^2:=\E[\|u\|_{L^2([0,T])}^2] + \E[\|\nabla u\|_{H\otimes L^2([0,T])}^2].$$ 
Alternatively, an element $u$ in $\D^{1,2}(\h)$ is identified with a stochastic process such that for almost avery $t$ in $[0,T]$, $u_t$ belongs to $\D^{1,2}$ and such that
$$ \E[\|\nabla u\|_{H\otimes L^2([0,T])}^2]=\E\left[\int_0^T \int_0^T |D_s u_t|^2 ds dt\right]<+\infty. $$
Hence we can assume that $\alpha$ belongs to $\D^{1,2}(\h)$. Thus by \cite[Theorem 3.1]{Sugita}, $\alpha$ satisfies (RAC) and (SGD), which entails in this setting that for any $h$ in $H$, there exists a $\h$-valued r.v. $\tilde{\alpha}_h$ such that $\tilde{\alpha}_h=\alpha$ in $\h$, $\P_0$-a.s., and for any $\varepsilon>0$ 
$$ \frac{\tilde{\alpha}_h \tr{\varepsilon h}-\tilde{\alpha}_h}{\varepsilon} =\varepsilon^{-1} \int_0^\varepsilon \langle \nabla \alpha \tr{s h}, h\rangle_H ds, \; \textrm{ in } \h, \; \P_0-a.s..  $$
Using Lemma \ref{lemma:UstunelZakai} we thus get that for any $r \in (p,2)$, it holds that:
\begin{align*}
\E\left[\int_0^T \left|\varepsilon^{-1} (\alpha_s \tr{\varepsilon h}-\alpha_s) \right|^r ds\right]&=\E\left[\int_0^T \left|\varepsilon^{-1} ((\tilde{\alpha}_h)(s) \tr{\varepsilon h}-(\tilde{\alpha}_h)(s)) \right|^r ds \right]\\
&\leq \E\left[\int_0^T \varepsilon^{-1} \int_0^\varepsilon |\langle \nabla \alpha_s \tr{u h}, h\rangle_H|^r du ds \right]\\
&\leq C \int_0^T \E\left[ |\langle \nabla \alpha_s, h\rangle_H|^p \right]^{r/p} ds \\
&\leq C \E\left[\int_0^T |\langle \nabla \alpha_s, h\rangle_H|^p ds \right]^{r/p}\\
&\leq C \|h\|_H^{r} \E\left[\|\nabla \alpha\|_{H\otimes \h}^p \right]^{r/p}<+\infty,
\end{align*}  
where we have used Cameron-Martin formula and similar computations to those of the proof of Lemma \ref{lemma:D12nec}, and $C$ denotes a positive constant which can differ from line to line. Hence, the family $\left(\int_0^T \left| \varepsilon^{-1} (\alpha_s \tr{\varepsilon h}-\alpha_s) - \langle \nabla \alpha_s,h\rangle_{H} \right|^p ds\right)_{\varepsilon \in (0,1)}$ is uniformly integrable. In addition, by Property (SGD), $\varepsilon^{-1} (\alpha \tr{\varepsilon h}-\alpha)$ converges in probability to $\langle \nabla \alpha, h\rangle_H$ (with respect to the norm $L^2([0,T])$) which implies that $\int_0^T \left|\varepsilon^{-1} (\alpha_s \tr{\varepsilon h}-\alpha_s) - \langle \nabla \alpha_s,h\rangle_{H}\right|^p ds$ converges in probability to $0$ as $\varepsilon$ goes to $0$, which provides the result.
\qed
\end{proof}

\section*{Acknowledgments}
Thibaut Mastrolia is grateful to R\'egion Ile-De-France for financial support. The equivalences between $(i)-(ii)-(iii)$ in Theorem \ref{th:newcarD} were stated and proved in a previous version of this paper. The potential equivalence with these properties and $(iv)$ has been suggested to us by Thomas Cass and we are very grateful to him for this remark that we have been able to prove here.

\end{document}